\documentclass[11pt]{article}
\usepackage{amsmath}
\usepackage{latexsym,amsfonts,amssymb,amsmath,longtable,amsthm}
\usepackage[a4paper,left=32mm,top=32mm,right=32mm, bottom=32mm]{geometry}
\usepackage{graphicx}
\usepackage[cp866]{inputenc}
\usepackage[english]{babel}
\usepackage{makeidx}
\usepackage{latexsym,amsfonts,amssymb,amsmath,longtable,amsthm}
\input amssym.def
\usepackage{latexsym}
\usepackage{graphicx}
\usepackage[all]{xy}
\usepackage{amsfonts}
\usepackage{amsmath}
\usepackage{mathrsfs}
\usepackage{color}
\usepackage{ulem}
\usepackage[table]{xcolor}
\usepackage[affil-it]{authblk}

\batchmode

\def\bfeta{\boldsymbol{\eta}}

\def\int{\intop}

\newcommand{\wotwo}{\mathring{W}^{1,2}}

\newcommand{\woq}{\mathring{W}^{k,q}}

\def\sR{{\mathbb{R}}}

\def\div{\hbox{\rm div}\,}

\newtheorem{theo}{\bf Theorem}
\newtheorem{lem}{\bf Lemma}
\newtheorem{remark}{Remark}
\newtheorem{defi}{\bf Definition}

\newcommand{\supp}{\mathop{\rm supp}}

\renewcommand{\div}{\mathop{\rm div}}

\newcommand{\F}{\mathbb{F}}

\newcommand{\ab}{\mathbf{a}}

\newcommand{\nn}{\mathbf{n}}

\title{Nonhomogeneous Boundary Value Problem for the Steady Navier-Stokes Equations in 2D Symmetric Domains with Several Outlets to Infinity }

\author[a]{K. Kaulakyt\.{e}}
 
\author[b]{W. Xue}

\affil[a]{Faculty of Mathematics and Informatics, Vilnius University\\
Naugarduko str. 24, LT-03225 Vilnius, Lithuania}

\affil[b]{Institute of Mathematics, University of Zurich\\
Winterthurerstrasse 190, CH-8057 Zurich, Switzerland}  

\date{ }

\batchmode

\begin{document}
\maketitle
\begin{abstract}
In  this paper we study the nonhomongeneous boundary value problem for the stationary Navier-Stokes equations in two dimensional symmetric domains with finitely many outlets to infinity. The domains may have no self-symmetric outlet ($\mathbb{V}$-type domain), one self-symmetric outlet ($\mathbb{Y}$-type domain) or two self-symmetric outlets ($\mathbb{I}$-type domain). We construct a symmetric solenoidal extension of the boundary value satisfying the Leray-Hopf inequality. After having such an extension, the nonhomogeneous boundary value problem is reduced to homogeneous one and the existence of at least one weak solution follows. Notice that we do not impose any restrictions on the size of the fluxes over the inner and outer boundaries. Moreover, the Dirichlet integral of the solution can be either finite or infinite depending on the geometry of the domains.
\end{abstract}

\setcounter{equation}{0}
\section{Introduction}

We study in this paper the steady nonhomogeneous boundary value problem to the Navier--Stokes equations in various types of two-dimensional symmetric unbounded domains with outlets to infinity. In \cite{KW} we considered this problem in domains consisting of only one outlet to infinity (either paraboloidal or channel-like). In this paper we generalize the method used in \cite{KW} in order to solve the problem in symmetric two-dimensional domains with several outlets to infinity (for instance, $\mathbb{V}$-type, $\mathbb{Y}$-type, $\mathbb{I}$-type domains).

Let us consider the steady Navier-Stokes equations with nonhomogeneous boundary conditions
\begin{equation}\label{prad0}\left\{\begin{array}{rcl}
-\nu \Delta{\bf u}+\big({\bf u}\cdot \nabla\big){\bf u} +\nabla p &
= & {\bf f}\qquad \hbox{\rm in }\;\;\Omega,\\
\div\,{\bf u} &  = & 0  \qquad \hbox{\rm in }\;\;\Omega,\\
 {\bf u} &  = & {\bf a} \qquad \hbox{\rm on }\;\;\partial\Omega
 \end{array}\right.\end{equation}
in a bounded domain with Lipschitz boundary $\partial \Omega$ consisting of $N$ disjoint components $\Gamma_j, j=1,...,N.$ The incompressibility of the fluid ($\div{\mathbf u} =  0$) implies a necessary compatibility condition for the solvability of problem \eqref{prad0}: 
\begin{equation}\label{nc}\begin{array}{l}
0=\int\limits_\Omega \div{\mathbf u}dx= \int\limits_{\partial\Omega}{\bf a}\cdot{\bf
n}\,dS=\sum\limits_{j=1}^N\int\limits_{\Gamma_j}{\bf a}\cdot{\bf n}\,dS=\sum\limits_{j=1}^N F_j,
\end{array}\end{equation}
where ${\bf n}$ is a unit vector of the outward  normal to $\partial\Omega$.
For a long time the solvability of problem \eqref{prad0} was proved only under the condition
\begin{equation}\label{3}\begin{array}{l}
{\F}_j=\int\limits_{\Gamma_j}{\bf a}\cdot{\bf n}\,dS=0,\qquad
j=1,2,\ldots,N,
\end{array}
\end{equation}
(e.g., \cite{Leray}, \cite{Lad}, \cite{Lad1} or \cite{VorJud})
or under the smallness assumptions on the fluxes $\F_j$ (e.g., \cite{BOPI}, \cite{Finn}, \cite{Fu}, \cite{Galdi1},
\cite{Kozono}),  or under certain symmetry assumptions on the domain
$\Omega$ and the boundary value ${\bf a}$ (e.g., {\cite{Amick},
\cite{Fu1}, \cite{FM}, \cite{Morimoto}, \cite{Pukhnachev},
\cite{Pukhnachev1}, \cite{Sazonov}, \cite{KPR1}).
In 1933 J. Leray formulated the fundamental question whether problem \eqref{prad0} can be solved only under the necessary compatibility condition  \eqref{nc}. This is so called Leray's problem which had been open for 80 years. Fortunately, recently Leray's problem was solved for a two dimensional multiply connected bounded domain (see \cite{KPR}, \cite{KPR0}, \cite{KPR5}). 
\\However, Leray's problem still remains open for some unbounded domains. For the non-symmetric domains with outlets to infinity problem \eqref{prad0}, \eqref{nc} was solved (see \cite{KAU}, \cite{KP}, \cite{Neustupa1}, \cite{Neustupa}) under the smallness assumption of the fluxes over the bounded components of the boundary (notice that there are no restrictions on the fluxes over the infinite parts of the boundary).
\\Next, there is series of papers by H. Fujita and H. Morimoto
(see \cite{Morimoto1}--\cite{Morimoto4}) where the authors solved problem \eqref{prad0} in symmetric two dimensional multiply connected domains $\Omega$ with
channel-like  outlets to infinity containing a finite number of
``holes" under the certain symmetry assumptions on the boundary value and external force. Furthermore,  in \cite{Morimoto1}--\cite{Morimoto4} the authors also assumed that the boundary value ${\bf a}$ is equal to zero on the
outer boundary and that in each outlet the flow tends to a Poiseuille flow which needs to be sufficiently small, i.e., although the fluxes over the boundary of each ``hole" may be arbitrarily large, but the sum of them has to be sufficiently small.
\\In \cite{KW} we study problem \eqref{prad0} in a two dimensional symmetric domain with one outlet to infinity (either paraboloidal or channel-like) under the same symmetry assumptions on the boundary value and the external force as in the papers of H. Fujita and H. Morimoto. Notice that we do not impose any restrictions on the fluxes over both: the inner and the outer boundaries. However, the technique we used in \cite{KW} works only for the domain with self-symmetric outlet and the outer boundary has to intersect the axis of the symmetry at least at one point. 

In this paper paper we study problem \eqref{prad0} in a class of symmetric multiply connected domains $\Omega\subset\mathbb{R}^2,$ having finitely many outlets to infinity, i.e., the domain may have self-symmetric and pairwise symmetric outlets. We assume that the boundary value and external force are symmetric functions, and the boundary value has a compact support. As in \cite{KW} we do not impose any restrictions on the sizes of the fluxes of ${\bf a}$ over the components of the inner and the outer boundaries. Under these conditions we prove 
the existence of at least one weak solution which can have  either finite or infinite Dirichlet integral, depending on the geometry of the domain. We give a constructive proof of the existence based on the special construction of a suitable extension  ${\bf A}$ of the boundary value ${\bf a}$ into the domain $\Omega.$ After we construct the suitable extension  ${\bf A},$ the existence of at least one weak solution can be proved on the same way as in the case of homogeneous boundary data (see \cite{LadSol3}, \cite{Sol3}). Moreover, this detailed proof is given in \cite{KW}. Therefore, we only give a general scheme of the existence proof and concentrate on the construction of a suitable extension of the boundary value.

\setcounter{equation}{0}
\setcounter{lem}{0}

\section{Main Notation and Auxiliary Results}

Vector valued functions are denoted by bold letters while function spaces for scalar and vector valued functions are not distinguished in notation.  

Let $\Omega$ be a domain in ${\mathbb{R}}^n$.
$C^{\infty}(\Omega)$ denotes the set of all infinitely
differentiable functions defined  on $\Omega$ and
$C_0^{\infty}(\Omega)$ is the subset of all functions from
$C^{\infty}(\Omega)$ with compact support in $\Omega$. For given
nonnegative  integers $k$ and $q > 1$, $L^q(\Omega)$ and $W^{k,
q}(\Omega)$ denote the usual Lebesgue and  Sobolev spaces; $W^{k-1/q,q}(\partial \Omega)$ is the trace space on
$\partial\Omega$ of functions from $W^{k, q}(\Omega)$;
$\woq(\Omega)$ is the closure of $C_{0}^{\infty}(\Omega)$ with respect to the
norm of $W^{k, q}(\Omega)$; for an unbounded domain $\Omega$ we write $u\in W^{k,q}_{
loc}(\overline{\Omega})$ if $u\in W^{k,q}(\Omega\cap B_R(0))$ for any $B_R(0)=\{x\in\sR^2:|x|\leq R\}.$

Let $D(\Omega)$ be the Hilbert space of vector valued functions formed as
the closure of $C_0^{\infty}(\Omega)$ with respect to the Dirichlet norm
$\|\textbf{u}\|_{ D(\Omega)}=\|\nabla\textbf{u}\|_{ L^2(\Omega)}$
induced by the scalar product
$$\begin{array}{l}
(\textbf{u},\textbf{v})=\int\limits_{\Omega}\nabla \textbf{u} :
\nabla\textbf{v} \ dx,
\end{array}
$$
where $\nabla\textbf{u} : \nabla\textbf{v}=\sum\limits^{n}_{j=1}
\nabla u_{j} \cdot \nabla v_{j}=\sum\limits_{j=1}^{n}
\sum\limits_{k=1}^{n} \dfrac{\partial u_{j}}{\partial x_{k}}
\dfrac{\partial v_{j}}{\partial x_{k}}.$ Denote  by
$J_0^{\infty}(\Omega)$  the set of all solenoidal ($\div {\bf u}=0$)
vector fields ${\bf u}$ from $C_0^{\infty}(\Omega)$. By $H(\Omega)$ we indicate the
space formed as the closure of $J_0^{\infty}(\Omega)$ with respect to the
Dirichlet norm.

Assume that $\Omega\subset\mathbb{R}^2$ is symmetric with respect to the $x_1$-axis, i.e.,
\begin{equation}\label{sd}
(x_1,x_2)\in\Omega \Leftrightarrow (x_1,-x_2)\in\Omega.
\end{equation}
The vector function ${\bf u}=(u_1, u_2)$ is called symmetric with respect to the $x_1$-axis if $u_1$ is an even function of $x_2$ and $u_2$ is an odd function of $x_2,$ i.e.,
\begin{equation}\label{ss}
u_1(x_1,x_2)=u_1(x_1,-x_2), \ \ \ \ u_2(x_1,x_2)=-u_2(x_1,-x_2).
\end{equation}
\\For any set of functions $V(\Omega)$ defined in the symmetric domain $\Omega$ satisfying \eqref{sd}, we denote by $V_S(\Omega)$ the subspace of symmetric functions from $V(\Omega).$

Below we use the well known results which are formulated in the following two lemmas.

\begin{lem}\label{hardy}(see \cite{Lad})
Let $\Pi\subset\sR^2$ be a bounded domain with Lipschitz boundary $\partial\Pi.$ Then for any ${\bf w}\in W^{1,2}(\Pi)$ with ${\bf w}\big|_{\mathcal{L}}=0,$ $\mathcal{L}\subseteq\partial\Pi,$ ${\rm meas (\mathcal{L})>0},$ the following inequality
\begin{equation}\label{hardy1}\begin{array}{l}
\int\limits_{\Pi}\dfrac{|{\bf w}|^2\,dx}{{\rm dist}^2(x,\mathcal{L})}\leq c\int\limits_{\Pi}|\nabla {\bf w}|^2\,dx
\end{array}
\end{equation}
holds.
\end{lem}

\begin{lem}\label{extension}(see \cite{Lad})
Let $\Pi\subset\sR^2$ be a bounded domain with Lipschitz boundary $\partial\Pi,$ $\mathcal{L}\subseteq\partial\Pi,$ ${\rm meas (\mathcal{L})>0}$ and the function ${\bf h}\in W^{1/2,2}(\partial\Pi)$ satisfies the conditions $\int\limits_{\mathcal{L}} {\bf h}\cdot {\bf n}\,dS=0,$ ${\rm supp}\,{\bf h}\subseteq \mathcal{L}.$ Then ${\bf h}$ can be extended inside $\Pi$ in the form
\begin{equation}\label{extension1}\begin{array}{l}
{\bf A}_*(x,\varepsilon)=\bigg(\dfrac{\partial(\chi(x,\varepsilon)\cdot {\bf E}(x))}{\partial x_2},- \dfrac{\partial(\chi(x,\varepsilon)\cdot {\bf E}(x))}{\partial x_1}\bigg),
\end{array}
\end{equation}
where ${\bf E}\in W^{2,2}(\Pi),$ $ \bigg(\dfrac{\partial {\bf E}(x)}{\partial x_2},- \dfrac{\partial{\bf E}(x)}{\partial x_1}\bigg)\bigg|_{\partial\Pi}={\bf h}$ and $\chi$ is a Hopf's type cut-off function, i.e., $\chi$ is smooth, $\chi(x,\varepsilon)=1$ on $\mathcal{L},$ ${\rm supp}\,\chi$ is contained in a small neighborhood of $\mathcal{L}$ and 
$$
|\nabla\chi(x,\varepsilon)|\leq \dfrac{\varepsilon\,c}{{\rm dist}(x,\mathcal{L})}.
$$
The constant $c$ is independent of $\varepsilon.$
\end{lem}

Let ${\mathcal M}$ be a closed set in $\sR^2.$ $\Delta_{{\mathcal
M}}(x)$  denotes  the regularized distance from the point  $x$ to
the set  ${\mathcal M}$. Notice that  $\Delta_{{\mathcal M}}(x)$ is an
infinitely differentiable function in $\sR^2\setminus {\mathcal M}$ and
the following inequalities
\begin{equation}\label{rd}\begin{array}{l}
a_1d_{\mathcal M} (x)\leq \Delta_{{\mathcal M}} (x)\leq a_2d_{\mathcal M} (x), \\
\\
|D^{\alpha}\Delta_{{\mathcal M}} (x)|\leq a_3 d_{\mathcal M}^{1-|\alpha|}
(x)\end{array}\end{equation}
hold, where $d_{\mathcal M}=dist(x,{\mathcal M})$ is the distance from
$x$ to ${\mathcal M}$, the positive constants $a_1, a_2$ and $a_3$ are independent of $\mathcal M$ (see \cite{Stein}).

\setcounter{equation}{0}
\setcounter{lem}{0}

\section{Problem Formulation and Solvability}

\subsection{Formulation of the Problem}

Let $\Omega\subset\mathbb{R}^2$ be an unbounded symmetric domain 
$$
\Omega = \Omega_0\cup D_1\cup...\cup D_N, \ \ \ D_j\cap D_k=\emptyset, \ \ j\neq k,
$$
where $\Omega_0 = \Omega \cap B_{R_0} (0)\subset B_{R_0} (0)$ is the bounded part of the domain $\Omega$ and the unbounded components $D_j, \ j=1,...,N,$ are called ''outlets to infinity." These outlets $D_j$ in some cartesian coordinate systems $z^{(j)}$ have the form
$$\begin{array}{l}
D_j=\{z^{(j)}\in\mathbb{R}^2: |z^{(j)}_2|<g_j(z^{(j)}_1), \ z^{(j)}_1>R_0\},
\end{array}$$  
where $z^{(j)}$ means the local coordinate system in the outlet $D_j$   and
$g_j(t)\geq R_0>0$ are functions satisfying the Lipschitz condition
$$|g_j(t_1)-g_j(t_2)|\leq L_j|t_1-t_2|, \ \ \ t_1,t_2\geq R_0. $$
Depending on the function $g_j$ each outlet $D_j$ may expand at infinity but not too much in order not to intersect each other. Notice that if the cross section of the outlet is constant, then we have channel-like outlet. Therefore, channel-like outlets are included as well.
\\Since we consider symmetric domains we may have pairs of outlets which are symmetric to each other (briefly we call them symmetric outlets) and self-symmetric outlets\footnote{If an outlet does not find another outlet symmetric to it, then this outlet itself is symmetric with respect to the $x_1$-axis and it is called self-symmetric outlet (see \cite{Morimoto4}).}.
\begin{defi}\label{domain} We call a symmetric domain $\Omega\subset\mathbb{R}^2$ an admissible domain if $\Omega$ satisfies the following assumptions
\\(i) the boundary $\partial\Omega$ is Lipschitz,
\\(ii) the bounded domain $\Omega_0$ has the form
\begin{equation*}\begin{array}{l}
\Omega_0=G_0\setminus \cup_{i=1}^I G_i,
\end{array}\end{equation*}
where $G_0$ and $G_i,\ \ i=1,...,I,$ are
bounded simply connected domains such that $\overline{G}_i \subset
G_0$. Each $\Gamma_i=\partial G_i, \ i=1,...,I,$ intersects the $x_1$- axis;
\\(iii) the boundary $\partial\Omega$ is composed of the inner boundarY $\cup_{i=1}^I \Gamma_i=\Gamma$ and the outer boundary $\partial\Omega\setminus\Gamma=\Gamma_0.$ The outer boundary $\Gamma_0$ consists of $N$ not connected unbounded components
$\Gamma_0^m, \ \ m=1,...,N,$ i.e., $\cup_{m=1}^N \Gamma_0^m=\Gamma_0,$
\\(iv) the outlets of $\Omega$ are of the following types:
$\mathbb{V}$ (no self-symmetric outlet) or 
$\mathbb{Y}$ (only one self-symmetric outlet)
or
$\mathbb{I}$ (two self-symmetric outlets).
\end{defi}

\begin{remark}
{\rm Notice that Definition\,\ref{domain} covers  not only $\mathbb{V},$ $\mathbb{Y}$ and $\mathbb{I}$ type outlets, but also all the possible combinations of them, i.e., domain $\Omega$ in general may have finite number of outlets to infinity.
}
\end{remark}
Below we use the following notation:
\begin{equation*}\begin{array}{l}
\Omega_{l}=\Omega_0\cup D_1^{(l)}\cup...\cup D_N^{(l)} \ \ {\rm with} \ \ D_j^{(l)}=\{z^{(j)}\in D_j: z^{(j)}_1<R_{j,\,l}\},\\
\\
R_{j,\,l+1}=R_{j,\,l}+\dfrac{g_j(R_{j,\,l})}{2L_j}, \ \ j=1,...,N.
\end{array}
\end{equation*}
We consider the following problem
\begin{equation}\label{prad00}\left\{\begin{array}{rcl}
-\nu \Delta{\bf u}+\big({\bf u}\cdot \nabla\big){\bf u} +\nabla p &
= & {\bf f}\qquad \hbox{\rm in }\;\;\Omega,\\
\div\,{\bf u} &  = & 0  \qquad \hbox{\rm in }\;\;\Omega,\\
 {\bf u} &  = & {\bf a} \qquad \hbox{\rm on }\;\;\partial\Omega\\
\int\limits_{\sigma_j(R)}{\bf u} \cdot {\bf n}\,dS & = & \mathbb{F}_j, \ \ \ j=1,...,N, \ R\geq R_0,
\end{array}\right.\end{equation}
where $\mathbb{F}_j, \ j=1,...,N,$ are the prescribed fluxes over the cross sections $\sigma_j(R)$ of the outlets $D_j,$ ${\bf n}$ is the unit vector of the normal to $\sigma_j.$
\\We  suppose that the boundary value ${\bf a}\in
W^{1/2,2}(\partial\Omega)$ has a compact support and we denote $\Lambda^m = \Gamma_0^{m}\cap\partial\Omega_0,$ $m=1,...,N.$
Let 
$$\begin{array}{l}
\int\limits_{\Gamma_i} {\bf a }\cdot {\bf n} \,dS=\mathbb{F}_i^{(inn)}, \ i=1,...,I \ \ \ \int\limits_{\Lambda^m} {\bf a }\cdot {\bf n} \, dS=\mathbb{F}_m^{(out)}, \ \ m=1,...,N\end{array}$$ 
be the fluxes of the boundary value ${\bf a}$ over the inner and outer boundaries, respectively. 
\\Notice that for $\Lambda^j$ and $\Lambda^k,$ $j\neq k,$ which are symmetric to each other, we have
$$\begin{array}{l}
\mathbb{F}_j^{(out)}=\int\limits_{\Lambda^j}{\bf a}\cdot {\bf n}\,dS=\int\limits_{\Lambda^k}{\bf a}\cdot {\bf n}\,dS=\mathbb{F}_k^{(out)}.
\end{array}$$
The necessary compatibility condition \eqref{nc} could be written as follows:
\begin{equation}\label{nc1}\begin{array}{l}
\sum\limits_{i=1}^I \mathbb{F}_i^{(inn)}+\sum\limits_{m=1}^N \mathbb{F}_m^{(out)}+\sum\limits_{j=1}^N \mathbb{F}_j=0.
\end{array}\end{equation}
\subsection{Solvability of Problem \eqref{prad00}}
\begin{defi}
Under a symmetric weak solution of problem \eqref{prad00} we understand a solenoidal vector field ${\bf u}\in W^{1,2}_{loc,S}(\overline{\Omega})$ satisfying the boundary condition ${\bf u}\big|_{\partial\Omega}={\bf a},$ the flux conditions 
$$\begin{array}{l}
\int\limits_{\sigma_j(R)}{\bf u}\cdot {\bf n}\,dS=\mathbb{F}_j, \ \ \ j=1,...,N, \ \ R>R_0
\end{array}$$
and the integral identity
\begin{equation}\label{ws33}\begin{array}{l}
\nu\int\limits_{\Omega} \nabla {\bf u} : \nabla\bfeta \, dx
-\int\limits_{\Omega} ({\bf u}\cdot\nabla)\bfeta\cdot{\bf u} \,
dx
=\int\limits_{\Omega}{\bf f}\cdot \bfeta\,dx, \ \ \ \
\ \forall \bfeta\in J^\infty_{0,S}(\Omega).
\end{array}\end{equation}
\end{defi}
The fundamental tool to solve the nonhomogeneous boundary value problem is to reduce it to the problem with homogeneous boundary conditions. Let ${\bf A}$ be a symmetric solenoidal extension of the boundary value ${\bf a}$ into $\Omega$ such that
\begin{equation}\label{ge}\begin{array}{l}
{\bf A}\big|_{\partial\Omega}={\bf a}, \ \ \ \int\limits_{\sigma_j(R)}{\bf A}\cdot {\bf n}\,dS=\mathbb{F}_j, \ \ j=1,...,N.
\end{array}\end{equation}
We put ${\bf u}={\bf v}+{\bf A}$ into identity \eqref{ws33} and look for the new unknown velocity field ${\bf v}\in W^{1,2}_{loc,S}(\overline{\Omega})$ satisfying the integral identity 
\begin{equation}\label{ws3}\begin{array}{l}
\nu\int\limits_{\Omega} \nabla {\bf v} : \nabla\bfeta \, dx -
\int\limits_{\Omega} (({\bf A}+{\bf v})\cdot\nabla)\bfeta\cdot{\bf v} \, dx
-\int\limits_{\Omega} ({\bf v}\cdot\nabla)\bfeta\cdot{\bf A} \,
dx\\
\\
=\int\limits_{\Omega} ({\bf A}\cdot\nabla)\bfeta\cdot{\bf A} \, dx -
\nu\int\limits_{\Omega} \nabla {\bf A} : \nabla\bfeta \, dx+\int\limits_{\Omega}{\bf f}\cdot \bfeta\,dx, \ \ \ \
\ \forall \bfeta\in J^\infty_{0,S}(\Omega).
\end{array}\end{equation}
and zero boundary and flux conditions:
$$\begin{array}{l}
\div {\bf  v}=0, \quad {\bf v}\big|_{\partial\Omega}=0,\quad
\int\limits_{\sigma_j(R)}{\bf v}\cdot{\bf n}\,dS = 0, \ \ j=1,...,N, \ \ R>R_0>0.
\end{array}
$$
\begin{remark}
{\rm Notice that the integral identity \eqref{ws3} remains valid for the non-symmetric functions $\bfeta\in J^\infty_0(\Omega)$ as well. It is well known that each function $\bfeta\in J^\infty_0(\Omega)$ can be decomposed to a sum $\bfeta=\bfeta_S+\bfeta_{AS},$ where $\bfeta_S$ is symmetric and $\bfeta_{AS}$ is antisymmetric, and it can be easily verified that all integrals in \eqref{ws3} vanish for $\bfeta=\bfeta_{AS}.$}
\end{remark}
The existence  of ${\bf v}$ satisfying \eqref{ws3} could be proved following the general scheme proposed by O. A. Ladyzhenskaya and V.A. Solonnikov (see \cite{LadSol3}, \ \cite{Sol3}). We give the sketch of this proof. Let us assume (as in \cite{Sol3}) that there is a sequence of bounded domains $\Omega_l$ such that $\Omega_l\subset\Omega_{l+1}$ and $\Omega_l$ exhausts $\Omega$ as $l\rightarrow +\infty.$ We shall find a solution to \eqref{ws3} as a limit solutions ${\bf v}^{(l)}\in H_S(\Omega_l),$ satisfying the following integral identity
\begin{equation}\label{ws3e}\begin{array}{l}
\nu\int\limits_{\Omega_l} \nabla {\bf v}^{(l)} : \nabla\bfeta \, dx -
\int\limits_{\Omega_l} (({\bf A}+{\bf v}^{(l)})\cdot\nabla)\bfeta\cdot{\bf v}^{(l)} \, dx
-\int\limits_{\Omega_l} ({\bf v}^{(l)}\cdot\nabla)\bfeta\cdot{\bf A} \,
dx\\
\\
=\int\limits_{\Omega_l} ({\bf A}\cdot\nabla)\bfeta\cdot{\bf A} \, dx -
\nu\int\limits_{\Omega_l} \nabla {\bf A} : \nabla\bfeta \, dx+\int\limits_{\Omega}{\bf f}\cdot \bfeta\,dx, \ \ \ \
\ \forall \bfeta\in J^\infty_{0,S}(\Omega_l).
\end{array}\end{equation}
This integral identity is equivalent to the operator equation:
\begin{equation}\label{oe}\begin{array}{l}
{\bf v}^{(l)}=\dfrac{1}{\nu}\mathcal{A}\,{\bf v}^{(l)}
\end{array}\end{equation}
with the compact operator $\mathcal{A}$ in the space $H_S(\Omega_l).$
The solvability of the operator equation \eqref{oe} can be obtained by applying the Leray-Schauder theorem, i.e., we need to show that all possible solutions of the operator equation 
\begin{equation}\label{oe1}\begin{array}{l}
{\bf v}^{(l,\lambda)}=\dfrac{\lambda}{\nu}\mathcal{A}\,{\bf v}^{(l, \lambda)}, \ \ \ \lambda\in [0,1],
\end{array}\end{equation}
are uniformly (with respect to $\lambda$) bounded. To prove this estimate we construct an extension ${\bf A}$ satisfying the following so called Leray-Hopf's inequality:
\begin{equation}\label{lh}\begin{array}{l}
\big|\int\limits_{\Omega_{l}}({\bf w} \cdot \nabla){\bf w} \cdot {\bf A}\,dx\big| \leq c\,\varepsilon \int\limits_{\Omega_{l}}|\nabla {\bf w}|^2\,dx
\end{array}\end{equation}
for every symmetric solenoidal function ${\bf w} \in W^{1,2}_{loc}(\overline{\Omega})$ with ${\bf w}|_{\partial\Omega}=0.$ 
\\If \eqref{lh} with an appropriate $\varepsilon$ is true, then we obtain the following estimate:
\begin{equation}\label{e}\begin{array}{l}
\int\limits_{\Omega_{l}}|\nabla\textbf{u}|^2\,dx
\leq c({\bf a}, \|{\bf f}\|_*)\, \bigg(1+\sum\limits_{j=1}^N\int\limits_{R_0}^{R_{j,l}}\dfrac{dx_1}{g_j^3(x_1)}\bigg),\end{array}\end{equation}
where constant $c({\bf a}, \|{\bf f}\|_*)$ is defined below in the Theorem\,3.1.
\\If $\int\limits_{R_0}^{+\infty}\dfrac{dx_1}{g_j^3(x_1)}<+\infty$ for every $j=1,...,N,$ then the right hand side of \eqref{e} is bounded by a constant uniformly independent of $l$ and we get for a limit vector function the integral identity \eqref{ws3}.
\\If there exists at least one number $j$ such that  $\int\limits_{R_0}^{+\infty}\dfrac{dx_1}{g_j^3(x_1)}=+\infty,$ then we cannot pass to a limit because the right hand side of \eqref{e} becomes infinite. Therefore, we need to control the Dirichlet integral of ${\bf v}^{(l)}$ over subdomains $\Omega_k\subset\Omega_l, \ \ k\leq l.$ To do this we need to apply the special techniques (so called estimates of Saint-Venant type) developed in \cite{LadSol3}, \ \cite{Sol3}.
Then we obtain the following estimate:
\begin{equation*}\begin{array}{l}
\int\limits_{\Omega_{k}}|\nabla\textbf{u}|^2\,dx
\leq c({\bf a}, \|{\bf f}\|_*)\, \bigg(1+\sum\limits_{j=1}^N\int\limits_{R_0}^{R_{j,k}}\dfrac{dx_1}{g_j^3(x_1)}\bigg), \ \ k\leq l.\end{array}\end{equation*}
This estimate ensures the existence of a subsequence $\{ {\bf v}^{(l_m)} \}$ which  converges weakly in $\wotwo(\Omega_k)$ and strongly in $L_4(\Omega_k), \ \ \forall k>0,$ and we can pass to a limit as $l_m\rightarrow +\infty.$ As a result we get for a limit vector function the integral identity \eqref{ws3}.

Therefore, a significant part in proving the existence of ${\bf v}$ concerns the construction of the extension ${\bf A}$ having properties \eqref{ge} and satisfying so called Leray-Hopf's inequality \eqref{lh}. Detailed existence proof for the domain having one self-symmetric outlet can be found in \cite{KW}.
\begin{theo}
Suppose that $\Omega\subset\sR^2$ is an admissible domain given in Definition\,\ref{domain}.  Assume that the boundary value $\textbf{a}$ is a symmetric field in $W^{1/2,2}(\partial\Omega)$ having a compact support, the external force ${\bf f}$ is a symmetric vector field such that for every $k$ the integral $\int\limits_{\Omega_k}{\bf f}\cdot\bfeta\,dx$ defines a  bounded functional on $H(\Omega_k).$ 
Then problem \eqref{prad00} admits at least one weak solution ${\bf u}={\bf A}+{\bf v}.$ Furthermore, if $
\int\limits_{R_0}^{+\infty}\dfrac{dx_1}{g^3_j (x_1)}<+\infty, \ \ j=1,...,N,
$ then the weak solution ${\bf u}$ satisfies the estimate
\begin{equation}\label{te1}\begin{array}{l}
\int\limits_{\Omega}|\nabla\textbf{u}|^2\,dx
\leq c({\bf a}, \|{\bf f}\|_*)\, \bigg(1+\sum\limits_{j=1}^N\int\limits_{R_0}^{+\infty}\dfrac{dx_1}{g^3_j(x_1)}\bigg)\end{array}\end{equation}
while if there is a number $j\in\{1,...,N \}$ such that $
\int\limits_{R_0}^{+\infty}\dfrac{dx_1}{g^3_j(x_1)}=+\infty,
$ then ${\bf u}$ satisfies
\begin{equation}\label{te}\begin{array}{l}
\int\limits_{\Omega_{k}}|\nabla\textbf{u}|^2\,dx
\leq c({\bf a}, \|{\bf f}\|_*)\, \bigg(1+\sum\limits_{j=1}^N\int\limits_{R_0}^{R_{j,k}}\dfrac{dx_1}{g_j^3(x_1)}\bigg),\end{array}\end{equation}
where 
\begin{equation*}\begin{array}{l}
\|{\bf f}\|_*=\sup\limits_{k\geq 1}\bigg(\big(1+\sum\limits_{j=1}^N\int\limits_{R_0}^{R_{j,k}}\dfrac{dx_1}{g_j^3(x_1)}\big)^{-1/2}\cdot \|{\bf f}\|_{H^{*}(\Omega_k)}\bigg), \ \
\|{\bf f}\|_{H^{*}(\Omega_k)}=\sup\limits_{\bfeta\in J^{\infty}_0(\Omega_k)}\dfrac{\big|\int\limits_{\Omega_k}{\bf f}\cdot\bfeta\,dx\big|}{\|\bfeta\|_{D(\Omega_k)}},
\end{array}\end{equation*}
$c({\bf a}, \|{\bf f}\|_*)=c\bigg(\|{\bf a}\|^2_{W^{1/2,\,2}(\partial\Omega)}+\| {\bf a}\|^4_{W^{1/2,\,2}(\partial\Omega)}\!+\!\|{\bf f}\|^2_*\bigg)$ and $c$ is independent of $k.$
\end{theo}

\setcounter{lem}{0}
\setcounter{equation}{0}
\section{Construction of the Extension}

We introduce the general idea of the construction of an extension ${\bf A}$ of the boundary value ${\bf a}.$ Notice that this general scheme works if the outer boundary intersects the $x_1$ axis at least at one point (for example, $\mathbb{V}$ and $\mathbb{Y}$ types domains). For the exception when a domain contains the $x_1$ axis, i.e., the outer boundary does not intersect the $x_1$ axis ($\mathbb{I}$ type domain) we combine the general scheme described below and the idea introduced in \cite{KW} (see subsection\,4.3 at the end of this section). In general case the extension ${\bf A}$ is constructed as the sum
$$\begin{array}{l}
{\bf A}={\bf B}^{(inn)}+\sum\limits_{m=1}^N{\bf B}^{(out)}_m+{\bf B}^{(flux)}.
\end{array}$$
In order to construct ${\bf B}^{(inn)},$ we ``remove" the fluxes  $\mathbb{F}^{(inn)}_i, \ \ i=1,...,I,$ (with the help of the virtual drain functions\footnote{The concept of virtual drain function was introduced by J. Fujita \cite{Fu1}.} ${\bf b}^{(inn)}_i$) to one of the the outer boundaries intersecting the $x_1$ axis, say $\Gamma_0^1,$ and then we extend the modified boundary value ${\bf a}-\sum\limits_{i=1}^I{\bf b}^{(inn)}_i$ which has zero fluxes on $\Gamma_i, \ i=1,...,I,$ into $\Omega.$ After this step we get the flux $\sum\limits_{i=1}^I\mathbb{F}^{(inn)}_i+\mathbb{F}^{(out)}_1$ on the outer boundary $\Gamma_0^1.$ Then by removing it to infinity and extending the modified boundary value ${\bf a}-\sum\limits_{i=1}^I{\bf b}^{(inn)}_i$ from $\Gamma_0^1$ into $\Omega$ we construct the extension ${\bf B}^{(out)}_1.$ Analogically we construct the rest of ${\bf B}^{(out)}_m, \ m=2,...,N.$ Notice that on the outer boundaries $\Gamma_0^m, \ m=2,...,N,$ we have the fluxes $\mathbb{F}^{(out)}_m, \ m=2,...,N,$ and after removing them to infinity, we extend the origin boundary value ${\bf a}$ from $\Gamma_0^m, \ m=2,...,N.$ Finally we need to compensate the fluxes over the cross sections of outlets to infinity, i.e., we construct ${\bf B}^{(flux)}$ satisfying zero boundary conditions and having the given flux over each cross section of all outlets. Notice that if a domain has only two outlets symmetric to each other or one self-symmetric outlet then we can not additionally prescribe the fluxes over the cross sections of the mentioned outlets, i.e., we take ${\bf B}^{(flux)}=0.$
\\By constructing functions ${\bf B}^{(out)}_m, \ \ m=1,...,N,$ we can choose arbitrarily an outlet to which we "remove" the fluxes. We choose the "widest" outlet in order to minimize the Dirichlet integral. Choosing different outlet we may get, in general, different solutions of problem \eqref{prad00}.

\begin{remark}
{\rm Construction of the vector field ${\bf B}^{(inn)}$ is based on the idea of J. Fujita (see \cite{Fu1}). In order to construct the vector fields ${\bf B}^{(out)}$ and ${\bf B}^{(flux)}$ we applied some ideas introduced by V. A. Solonnikov (see \cite{Sol3}).}
\end{remark}

\begin{remark}
{\rm  Notice that the symmetry assumptions are crucial for the construction of the vector field ${\bf B}^{(inn)},$ satisfying the Leray-Hopf inequality. If the fluxes over connected components of the boundary are nonzero, there is a counterexample (see \cite{Tak}) 
showing  that for a general  bounded domain  it is impossible 
to extend the boundary value ${\bf a}$  as a solenoidal vector field ${\bf A},$ satisfying the Leray-Hopf 
inequality \eqref{lh}.   However, such an extension can be constructed under certain symmetry assumptions (see \cite{Fu1}). 
As soon as we have a suitable extension, we can apply the existence proof given in section\,3.2.}

\end{remark}

\subsection{$\mathbb{V}$ Type Domain}
In this subsection we consider 
\begin{equation}\label{prad}\left\{\begin{array}{rcl}
-\nu \Delta{\bf u}+\big({\bf u}\cdot \nabla\big){\bf u} +\nabla p &
= & {\bf f}\qquad \hbox{\rm in }\;\;\Omega,\\
\div\,{\bf u} &  = & 0  \qquad \hbox{\rm in }\;\;\Omega,\\
 {\bf u} &  = & {\bf a} \qquad \hbox{\rm on }\;\;\partial\Omega,
 \end{array}\right.\end{equation}
in the domain $\Omega=\Omega_0\cup D_1\cup D_2,$ where $D_1,$ $D_2$ is a pair of symmetric outlets (see Fig.\,1). According to Definition\,\ref{domain} in this special case the outer boundary consists of two unbounded disjoint connected components $\Gamma_0^1$ and $\Gamma_0^2.$ The fluxes over the outer boundaries are
$$\begin{array}{l}
\int\limits_{\Lambda^m}{\bf a}\cdot {\bf n}\,dS=\mathbb{F}_m^{(out)}, \ \ \ m=1,2,
\end{array}$$
and the fluxes over the inner boundaries are
$$\begin{array}{l}
\int\limits_{\Gamma_i}{\bf a}\cdot {\bf n}\,dS=\mathbb{F}_i^{(inn)}, \ \ \ i=1,...,I.
\end{array}$$
The necessary compatibility condition \eqref{nc1} can be rewritten as
\begin{equation*}\begin{array}{l}
\mathbb{F}^{(inn)}+\mathbb{F}^{(out)}+\mathbb{F}_1=0,
\end{array}\end{equation*}
where $$\mathbb{F}^{(inn)}=\sum\limits_{i=1}^I \mathbb{F}_i^{(inn)}, \ \ \mathbb{F}^{(out)}=\sum\limits_{m=1}^2 \mathbb{F}_m^{(out)}.$$
Hence, in this case we cannot prescribe additionally the flux $\mathbb{F}_1,$ i.e.,
$$\begin{array}{l}
\mathbb{F}_1=-\big(\mathbb{F}^{(inn)}+ \mathbb{F}^{(out)}\big)
\end{array}$$
and the extension ${\bf B}^{(flux)}$ can be taken equal to zero.

\subsubsection{Construction of the extension ${\bf B}^{(inn)}$}

We construct a symmetric solenoidal extension ${\bf B}^{(inn)}$ of the boundary value $\bf a$ from the inner boundary. This extension is constructed to satisfy the Leray-Hopf inequality. In a bounded domain if the fluxes over the connected components of the boundary do not vanish, one cannot expect that there exists such an extension (see the counterexample in \cite{Tak}). However, under the symmetry assumptions such an extension can be constructed. We follow the idea of Fujita \cite{Fu1} for a bounded symmetric domain. 
\begin{lem}\label{B_inn}
Let $\ab \in W^{1/2,2}(\Gamma)$ be a symmetric function. Then for $\forall \varepsilon > 0$ there exists a symmetric solenoidal extension ${\bf B}^{(inn)}$ in $\Omega$ satisfying the Leray-Hopf inequality, i.e., for every symmetric solenoidal ${\bf w} \in W^{1,2}_{loc}(\overline{\Omega})$ with ${\bf w}|_{\partial\Omega}=0$ the following estimate\footnote{Notice that the integral in \eqref{binn00} over $\Omega_{k+1}\setminus\Omega_k$ is equal to zero since ${\bf B}^{(inn)}=0$ in $\Omega_{k+1}\setminus\Omega_k.$}
\begin{equation}\label{binn00}\begin{array}{l}
\big|\int\limits_{\Omega_{k+1}}({\bf w} \cdot \nabla){\bf w} \cdot {\bf B}^{(inn)} dx\big| \leq c\,\varepsilon \int\limits_{\Omega_{k+1}}|\nabla {\bf w}|^2\,dx
\end{array}\end{equation}
holds.
\end{lem}

\begin{proof}
In order not to lose the main idea in technical details, we assume that there is only one connected component of the inner boundary, i.e. $\Gamma_1.$ The same construction works for domains with finitely many inner boundaries (see Remark\,\ref{r1} at the end of this section).
The proof is composed of two parts. We construct firstly a solenoidal function in $\Omega_0$ which ``takes" the flux of the boundary value ${\bf a}$ from $\Gamma_1$ to $\Gamma^1_0$. Then the second part is to construct the extension of the boundary  value, having zero flux on $\Gamma_1,$ to $\Omega.$\\
Let $0<\kappa<1/2$ be a number and $\beta_\kappa(t)\in C_0^\infty(-\infty,+\infty)$ be an even function such that $\beta_\kappa(t)\leq \dfrac{1}{t},\ {\rm for} \ 0<t<+\infty$ and
\begin{equation*}\beta_\kappa(t)=\begin{cases}
0, \ \  \ 1\leq t < +\infty,\\
\dfrac{1}{t}, \ \  \ \kappa \leq t \leq 1/2.
\end{cases}\end{equation*}
Define $y_\kappa= \int\limits_{-\infty}^\infty \beta_\kappa(t) dt$. Note that $y_\kappa \to +\infty$ as $\kappa \to 0$. Let $\delta$ be another small positive number. Then we define a smooth positive function $s(t)=\dfrac{1}{y_\kappa\, \delta}\beta_\kappa\Big(\dfrac{t}{\delta}\Big)$. Note that $\supp s\subset [-\delta,\delta]$. Moreover, we have 
$$\begin{cases}
\int\limits_{-\infty}^\infty s(t)\,dt=\int\limits_{-\delta}^{\delta} s(t)\,dt=1,\\
0\leq s(t)\leq \dfrac{1}{y_\kappa \delta}\cdot\dfrac{\delta}{|t|}=\dfrac{1}{y_\kappa |t|},\quad t\neq 0.
\end{cases}
$$
Therefore, we have
\begin{equation}\label{aux}
\sup_t |t|s(t) \to 0 \quad \rm{as} \ \kappa \to 0.
\end{equation}
Now we construct a thin strip $\Upsilon=[x_0+\eta,x_1-\eta]\times [-\delta, \delta],$ where $\eta$ is a small positive number, connecting the inner boundary $\Gamma_1$ and the outer boundary $\Gamma_0^1$ (see Fig.\,1).
\begin{figure}[ht!]
\centering
\includegraphics[scale=0.35]{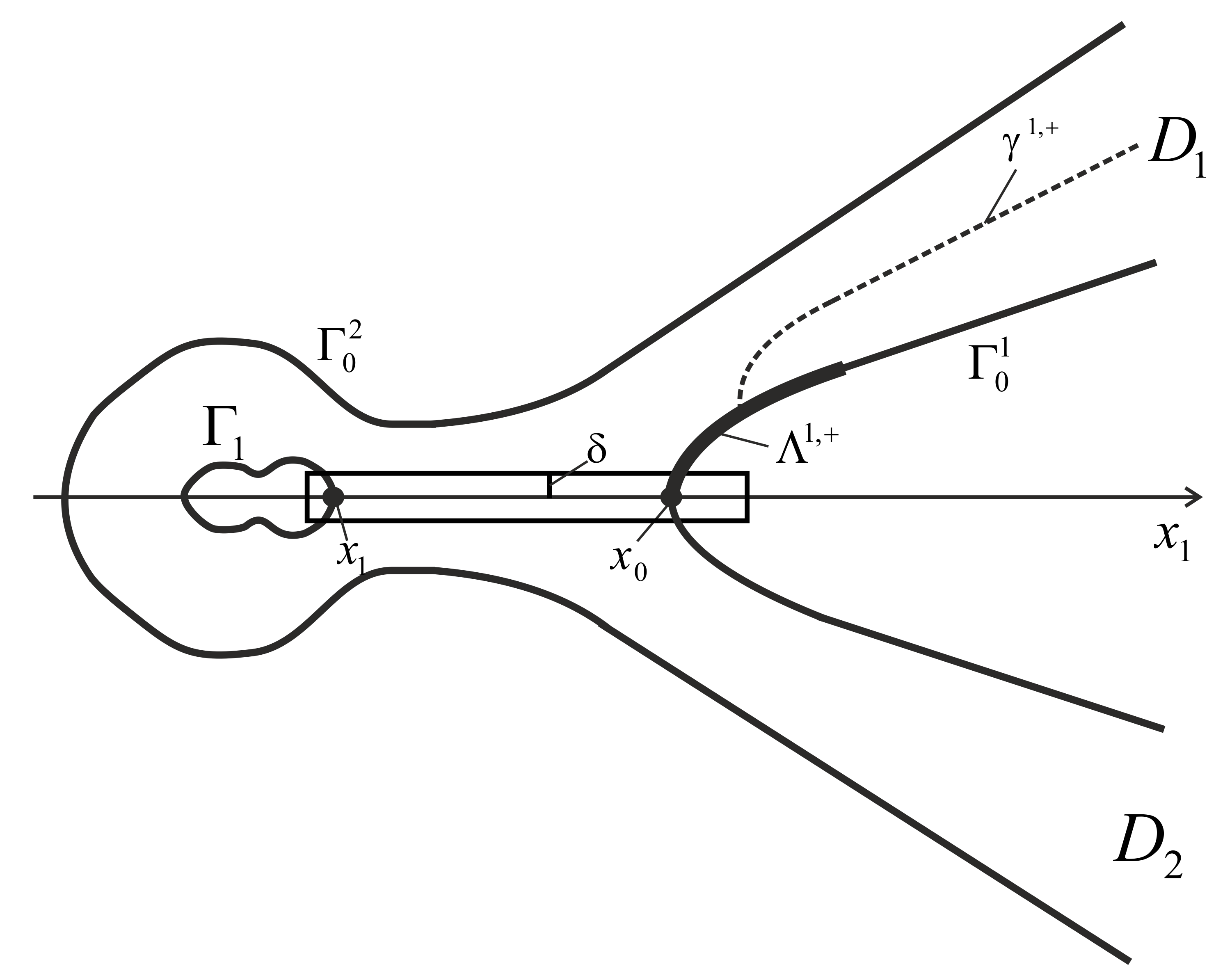}
\caption{Domain $\Omega$}
\end{figure}
\\We define a solenoidal function compactly supported in $\Upsilon$ which takes the flux from the inner boundary $\Gamma_1$ to the outer boundary $\Gamma^1_0:$ 
\begin{equation}\label{bin}\begin{array}{l}
{\bf b}^{(inn)}(x)=\begin{cases}
(-\mathbb{F}_1^{(inn)}s(x_2),0) \quad \rm{in} \ \Upsilon,\\
\\
(0,0) \quad \rm{in} \ \overline{\Omega} \setminus \Upsilon.
\end{cases}
\end{array}\end{equation}
Since the vector field ${\bf b}^{(inn)}$ is solenoidal and vanishes on the upper and lower boundaries of $\Upsilon$, we have

$$
\begin{array}{l}
0=\int\limits_{\Upsilon\cap G_1} {\rm div}\,{\bf b}^{(inn)}\,dx=\int\limits_{\partial(\Upsilon \cap  G_1)}{\bf b}^{(inn)}\cdot \nn \,dS\\
\\
=\int\limits_{\Upsilon \cap \Gamma_1}{\bf b}^{(inn)}\cdot \nn\, dS+\int\limits_{(x_1-\eta)\times[-\delta,\delta]}{\bf b}^{(inn)}\cdot \nn \,dS\\
\\
=\int\limits_{\Upsilon \cap \Gamma_1}{\bf b}^{(inn)}\cdot \nn\, dS+
\mathbb{F}_1^{(inn)}\int\limits_{-\delta}^{\delta} s(x_2)\,dx_2
=\int\limits_{\Upsilon \cap \Gamma_1}{\bf b}^{(inn)}\cdot \nn \,dS+\mathbb{F}_1^{(inn)}.
\end{array}$$
Notice that the outward normal vector on $\partial(\Upsilon\cap G_1)|_{\Gamma_1}$ shows the opposite direction than the one on $\Gamma_1.$ Therefore, it follows that 
$$\begin{array}{l}
\int\limits_{\Gamma_1}{\bf b}^{(inn)}\cdot \nn dS=\mathbb{F}_1^{(inn)}.
\end{array}$$
Let $${\bf h}^{(inn)}={\bf a}\big|_{\Gamma_1}-{\bf b}^{(inn)}\big|_{\Gamma_1}.$$
Then we have
\begin{equation}\label{h0}\begin{array}{l}
\int\limits_{\Gamma_1}{\bf h}^{(inn)}\cdot {\bf n}\,dS=\int\limits_{\Gamma_1}(\ab-{\bf b}^{(inn)})\cdot {\bf n}\,dS=0.
\end{array}\end{equation}
By the condition \eqref{h0}, there exists (see Lemma\,\ref{extension}) an extension ${\bf b}_0^{(inn)}$ of the function ${\bf h}^{(inn)}$ such that ${\rm supp}\,{\bf b}_0^{(inn)}$ is contained in a small neighborhood of $\Gamma_1,$
\begin{equation*}\label{b0}\begin{array}{l}
{\rm div}\,{\bf b}_0^{(inn)}=0, \ \ {\bf b}_0^{(inn)}|_{\Gamma_1}={\bf h}^{(inn)},
\end{array}
\end{equation*}
and ${\bf b}_0^{(inn)}$ satisfies the Leray--Hopf inequality
\begin{equation}\label{binn0}\begin{array}{l}
\big|\int\limits_{\Omega_{k+1}}({\bf w} \cdot \nabla){\bf w} \cdot {\bf b}^{(inn)}_0 dx\big| \leq c\,\varepsilon \int\limits_{\Omega_{k+1}}|\nabla {\bf w}|^2\,dx.
\end{array}\end{equation}
\\Notice that the vector field ${\bf b}^{(inn)}_0$ is not necessary symmetric. However, since the boundary value ${\bf h}^{(inn)}$ is symmetric, ${\bf b}^{(inn)}_0$ can be symmetrized to $\tilde{\bf b}^{(inn)}_{0}= (\tilde{b}^{(inn)}_{0,1},\tilde{b}^{(inn)}_{0,2})$ as follows:
\begin{equation}\label{symm}\begin{array}{l}
\tilde{b}^{(inn)}_{0,1}(x,\varepsilon)= \dfrac{1}{2}\bigg(b^{(inn)}_{0,1} (x_1,x_2,\varepsilon) +  b^{(inn)}_{0,1} (x_1,-x_2,\varepsilon)\bigg),\ \ x\in\Omega\\
\\
\tilde{b}^{(inn)}_{0,2}(x,\varepsilon)=\dfrac{1}{2}\bigg(b^{(inn)}_{0,2} (x_1,x_2,\varepsilon) -  b^{(inn)}_{0,2} (x_1,-x_2,\varepsilon)\bigg), \ \ x\in\Omega.
\end{array}
\end{equation}
\\Set 
$$\mathbf B^{(inn)}={\bf b}^{(inn)}\big|_{\Upsilon\cap\overline{\Omega}}+\tilde{{\bf b}}^{(inn)}_0.$$
It remains to prove that ${\bf b}^{(inn)}|_{\Upsilon\cap\overline{\Omega}}$ satisfies the Leray-Hopf inequality. For the sake of simplicity we keep the notation ${\bf b}^{(inn)}$ instead of ${\bf b}^{(inn)}|_{\Upsilon\cap\overline{\Omega}}$ for the rest of this paper and we extend it by zero into the whole domain $\Omega.$

Let ${\bf w} =(w_1,w_2)\in W^{1,2}_{loc}(\overline{\Omega}), \ {\bf w}|_{\partial\Omega}=0,$ be a symmetric and solenoidal vector field. Then we use the well known following equality
\begin{equation}\begin{array}{l}{\label{rot}}
({\bf w} \cdot \nabla){\bf w}=\nabla (\dfrac{1}{2}|{\bf w}|^2)+(\dfrac{\partial w_2}{\partial x_1}-\dfrac{\partial w_1}{\partial x_2})(-w_2,w_1).
\end{array}\end{equation}
\\Due to \eqref{rot} and \eqref{bin}, we have
\begin{equation*}\begin{array}{rcl}
\big|\int\limits_{\Omega_{k+1}}({\bf w} \cdot \nabla){\bf w}\cdot{\bf b}^{(inn)}dx\big|&=|\F_1^{(inn)}|\int\limits_{\Upsilon\cap\overline{\Omega}}\big|    \dfrac{\partial w_2}{\partial x_1}-\dfrac{\partial w_1}{\partial x_2}\big|\dfrac{1}{|x_2|}|w_2|\,|x_2|\,s(x_2)dx\\
\\
&\leq |\F_1^{(inn)}| \sup\limits_{x_2}(|x_2|s(x_2))\int\limits_{\Upsilon\cap\overline{\Omega}}
\big|\dfrac{\partial w_2}{\partial x_1}-\dfrac{\partial w_1}{\partial x_2}\big|\dfrac{|w_2|}{|x_2|}dx.
\end{array}\end{equation*}
By applying the H{\"o}lder inequality, the expression above is less than 
$$\begin{array}{l}|\F_1^{(inn)}| \sup\limits_{x_2}(|x_2|s(x_2))\,(\int\limits_{\Upsilon\cap\overline{\Omega}}|\dfrac{\partial w_2}{\partial x_1}-\dfrac{\partial w_1}{\partial x_2}|^2 dx)^{1/2} \cdot(\int\limits_{\Upsilon\cap\overline{\Omega}} \dfrac{|w_2|^2}{|x_2|^2}dx)^{1/2}.\end{array}$$
Furthermore,
$$\begin{array}{l}\int\limits_{\Upsilon\cap\overline{\Omega}} |\dfrac{\partial w_2}{\partial x_1}-\dfrac{\partial w_1}{\partial x_2}|^2 dx\leq 2\int\limits_{\Omega_{k+1}}|\nabla {\bf w}|^2 dx\end{array},$$
and applying \footnote{Here we use the fact that due to the symmetry assumptions,  the second component $w_2$ of ${\bf w}$ vanishes on the $x_1$-axis (in trace sense).} the Hardy inequality we obtain the following estimate
$$\begin{array}{l}\int\limits_{\Upsilon\cap\overline{\Omega}} \dfrac{|w_2|^2}{|x_2|^2}dx\leq \int\limits_{\Omega_{k+1}}|\nabla {\bf w}|^2 dx.\end{array}$$
\\Since $\sup\limits_{0\leq t \leq \eta}(|t|s(t))$ goes to zero as $\kappa$ goes to zero, we can choose $\kappa$ so small that $\sup\limits_{x_2}(|x_2|s(x_2))$ is less than $\dfrac{\varepsilon}{2\,\sqrt{2}}.$
Finally, we obtain the Leray-Hopf inequality \eqref{binn00}. 
\end{proof}
\begin{figure}[ht!]
\centering
\includegraphics[scale=0.35]{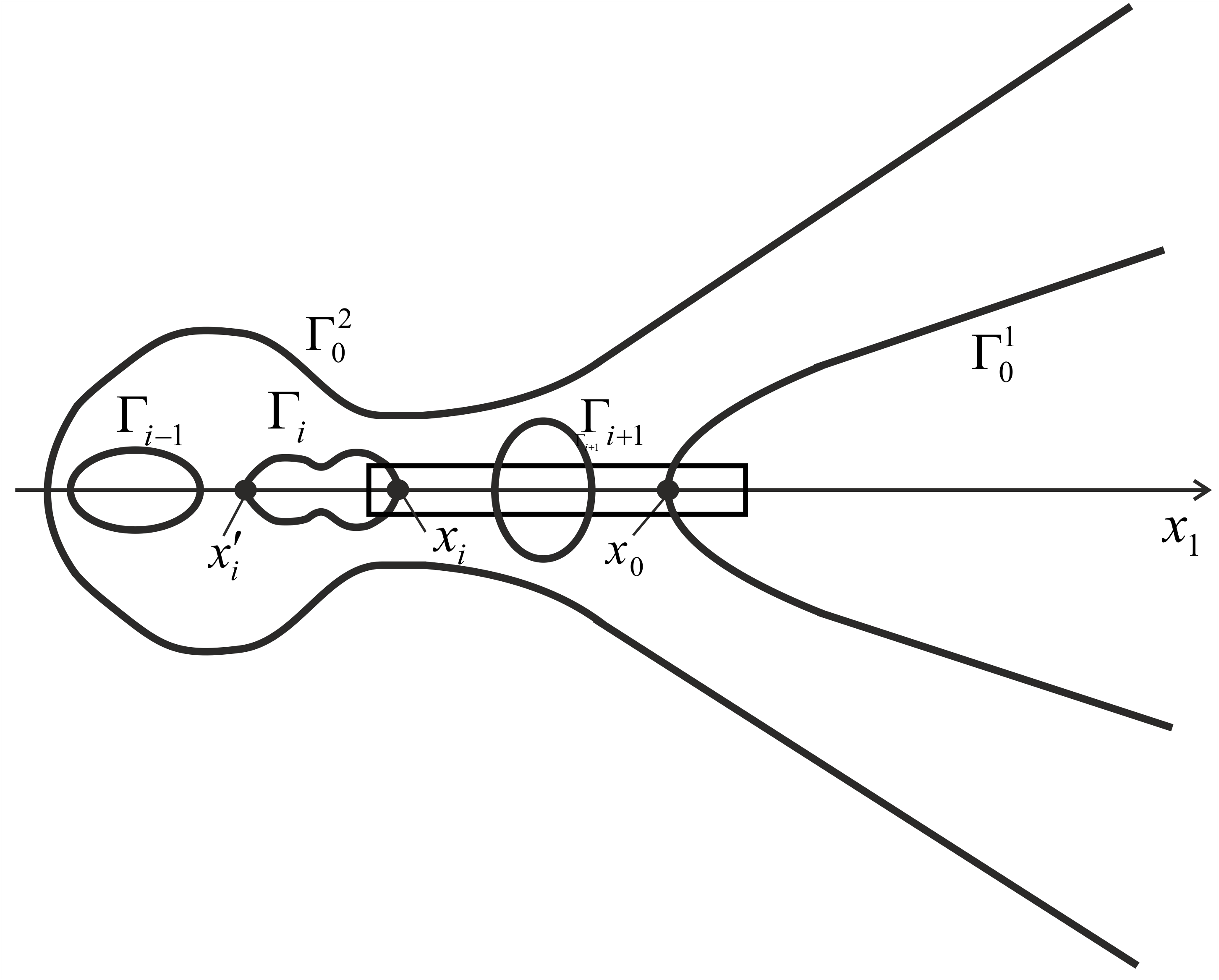}
\caption{Domain $\Omega$}
\end{figure}
\begin{remark}\label{r1}
{\rm The same construction works for domains with finitely many inner boundaries. For each $\Gamma_i$, $i=1,\dots,I,$ the $x_1$-axis intersects $\Gamma_i$ at two points $(x_i,0)$ and $(x_i',0),$ where $x_i>x_i'.$
In the same manner as before we can construct a thin strip $[x_0+\eta, x_i-\eta]\times[-\delta,\delta]$ connecting $\Gamma^1_0$ and $\Gamma_i, \ i=1,...,I$ (see Fig.\,2). We define a function ${\bf b}^{(inn)}_i$ on the thin strip as we did for $\Gamma_1,$ which satisfies
$$\begin{array}{l}
\int\limits_{\Gamma_j} {\bf b}^{(inn)}_i\cdot \nn\,dS=\begin{cases}
\mathbb{F}^{(inn)}_i, \ j=i,\\
0, \ j\neq i, \ j=1,...,I.
\end{cases}
\end{array}$$
\\Note that the flux of ${\bf b}_i^{(inn)}$ across $\Gamma_j$ cancel each other for $j<i$ and ${\bf b}_i^{(inn)}$ vanishes on $\Gamma_j$ for $j>i.$
\\Set ${\bf b}^{(inn)}=\sum\limits_{i=1}^I {\bf b}_{i}^{(inn)}.$ Then
$
\int\limits_{\Gamma_i} ({\bf a} - {\bf b}^{(inn)})\cdot \nn\,dS=0,\quad i=1,\dots,I.$
\\Using the usual technique we can find a symmetric solenoidal extension $\tilde{{\bf b}}_{0}^{(inn)}$ of the function $\big({\bf a} - {\bf b}^{(inn)}\big)|_{\cup_{i=1}^I\Gamma_i}$ satisfying the Leray-Hopf inequality. Then
$${\bf B}^{(inn)}={\bf b}^{(inn)}+\tilde{{\bf b}}_{0}^{(inn)}$$
is a suitable extension of the boundary value ${\bf a}$ from the inner boundaries $\Gamma_1,...,\Gamma_I.$}
\end{remark}

\subsubsection{Construction of ${\bf B}^{(out)}$}

After the construction of the extension ${\bf B}^{(inn)}$ of the boundary value ${\bf a}$ from the inner boundaries we need to construct an extension of the outer boundary value. Since the outer boundary consists of two not connected disjoint components $\Gamma_0^1$ and $\Gamma_0^2,$ we have to construct two vector fields ${\bf B}_1^{(out)}$ and ${\bf B}_2^{(out)}$ which extend the boundary value from the connected components $\Gamma_0^{1}$ and $\Gamma_0^{2},$ respectively. More precisely, ${\bf B}_1^{(out)}$ extends the modified boundary value ${\bf a}-{\bf b}^{(inn)}$ from the connected component $\Gamma_0^{1}$ and ${\bf B}_2^{(out)}$ extends the origin boundary value ${\bf a}$ from the connected component $\Gamma_0^{2}.$
Hence, the extension of the outer boundary $\Gamma_0^1\cup\Gamma_0^2$ has to be constructed as the sum
\begin{equation}\begin{array}{l}\label{B}
{\bf B}^{(out)}={\bf B}_1^{(out)}+{\bf B}_2^{(out)},\end{array}\end{equation}
where $$\begin{array}{l}
{\bf B}_1^{(out)}={\bf a}-{\bf b}^{(inn)} \ \ {\rm on} \ \ \Gamma_0^1\cap\Omega_0, \ \ \ {\bf B}_1^{(out)}=0 \ \ {\rm on} \ \ \partial\Omega\setminus (\Gamma_0^1\cap\Omega_0),
\end{array}$$
$$\begin{array}{l}
{\bf B}_2^{(out)}={\bf a} \ \ {\rm on} \ \ \Gamma_0^2\cap\Omega_0, \ \ \ {\bf B}_2^{(out)}=0 \ \ {\rm on} \ \ \partial\Omega\setminus (\Gamma_0^2\cap\Omega_0).
\end{array}$$
Let  $$\Omega^+ = \{x\in\Omega: x_2>0 \}, \ \ \Gamma_0^{1,+}=\Gamma_0^1\cap\Omega^+.$$
Let us start with the construction of ${\bf B}_1^{(out)}.$ Firstly we construct ${\bf b}_{+}^{(out)}$ in the domain $\Omega^+.$ Take any point $x_+\in\Lambda^{1,+}=\Lambda^1\cap\overline{\Omega}_+\subset \partial\Omega_0\cap\Gamma^{1,+}_0,$ where $\Lambda^1$ is the support of the boundary value ${\bf a}$ on $\Gamma^1_0.$ Introduce a smooth semi-infinite simple curve $\gamma^{1,+}\subset\Omega^+$ intersecting $\Gamma_0^{1,+}$ at the point $x_+$ such that the distance from the curve $\gamma^{1,+}$ to $\Gamma_0^{1,+}\setminus\Lambda^{1,+}$ is not less than the positive number $d_0.$ 
\\Then define a Hopf cut-off function by the formula
\begin{equation}\label{xi}\begin{array}{l}
\xi_+(x,\varepsilon)=\Psi\bigg(\varepsilon\,{\rm ln}\dfrac{\rho(\Delta_{\gamma^{1,+}}(x))}{\Delta_{\partial\Omega^+\setminus\Lambda^{1,+}}(x)}\bigg),
\end{array}
\end{equation}
where $0\leq\Psi\leq 1$ and $\rho$ are the smooth monotone cut-off functions:
\begin{equation}\label{fi}
\begin{array}{l}
\Psi(t)=\begin{cases}
0, \ \ t\leq 0,\\
1, \ \ t\geq 1,
\end{cases}
\end{array}
\end{equation}

\begin{equation}\label{ro}
\begin{array}{l}
\rho(\tau)=\begin{cases}
\dfrac{a_1}{2}d_0, \ {\rm for} \ \tau\leq \dfrac{a_2}{2}d_0,\\
\tau, \ {\rm for} \ \tau\geq a_2 d_0,
\end{cases}
\end{array}
\end{equation}
with the constants $a_1, a_2$  from the inequalities \eqref{rd}.

\begin{lem}\label{xl}
The function $\xi_+ (x,\varepsilon)$ is equal to zero at those points of $\Omega^+$ where $\rho(\Delta_{\gamma^{1,+}} (x))\leq\Delta_{\partial\Omega^+\setminus\Lambda^{1,+}} (x),$ while the $d_0/2$-neighborhood of the curve $\gamma^{1,+}$ is contained in this set. At those points where 
$\Delta_{\partial\Omega^+\setminus\Lambda^{1,+}} (x)\leq e^{-1/\varepsilon}\rho(\Delta_{\gamma^{1,+}} (x)),$ the function $\xi_+ (x,\varepsilon)=1.$ The following inequalities
\begin{equation}\label{xil}\begin{array}{l}
\Big|{\dfrac{\partial\xi_+(x,\varepsilon)}{\partial x_k}}\Big|\leq\dfrac{
c\,\varepsilon}{\Delta_{\partial\Omega^+\setminus\Lambda^{1,+}} (x)},\qquad \Big|\dfrac{\partial^2\xi_{+}(x,\varepsilon)}{\partial
x_k\partial x_l}\Big|\leq\dfrac{c\,\varepsilon}{\Delta^2_{\partial\Omega^+\setminus\Lambda^{1,+}}(x)}
\end{array}\end{equation}
hold.
\end{lem}

\begin{proof}
Due to the definition of $\Psi,$ the function $\xi_+ (x,\varepsilon)=0$ if $\rho(\Delta_{\gamma^{1,+}})\leq\Delta_{\partial\Omega^+\setminus\Lambda^{1,+}}$ and $\xi_+ (x,\varepsilon)=1$ if $\Delta_{\partial\Omega^+\setminus\Lambda^{1,+}}\leq e^{-1/\varepsilon}\rho(\Delta_{\gamma^{1,+}}).$ Let $d_{\gamma^{1,+}}={\rm dist}(x,\gamma^{1,+})$ and $d_{\partial\Omega^+\setminus\Lambda^{1,+}}={\rm dist}(x,\Omega^+\setminus\Lambda^{1,+}).$ If $d_{\gamma^{1,+}}\leq\dfrac{d_0}{2},$ then $d_{\partial\Omega^+\setminus\Lambda^{1,+}}\geq d_0-d_{\gamma^{1,+}}\geq\dfrac{d_0}{2}$ and $\Delta_{\gamma^{1,+}}\leq a_2 d_{\gamma^{1,+}}\leq \dfrac{a_2 d_0}{2}.$ Therefore if $\Delta_{\gamma^{1,+}}\leq \dfrac{a_2 d_0}{2},$ then 
$\rho(\Delta_{\gamma^{1,+}})=a_1\dfrac{d_0}{2}\leq a_1 d_{\partial\Omega^+\setminus\Lambda^{1,+}}\leq \Delta_{\partial\Omega^+\setminus\Lambda^{1,+}}$ and $\xi_+ (x,\varepsilon)=0$ (the constants $a_1$ and $a_2$ are from the inequalities \eqref{rd}).

The estimates \eqref{xil} follow by direct computation using the definitions of the functions $\Psi,$ $\rho,$ the properties of the regularized distance (see inequality \eqref{rd}) and the fact that ${\rm supp}\nabla\xi_+$ is contained in the set where $\Delta_{\partial\Omega^+\setminus\Lambda^{1,+}}\leq \rho(\Delta_{\gamma^{1,+}})$ (for details of computation see the proof of Lemma\,2 in \cite{PilSol}).
\end{proof}

Since $\gamma^{1,+}$ decomposes $\Omega^+$ into two parts situated above and below $\gamma^{1,+},$ we define
$\tilde{\xi}_+(x,\varepsilon)=\xi_+(x,\varepsilon)$ above  $\gamma^{1,+}$ and $\tilde{\xi}_+(x,\varepsilon)=0$ below  $\gamma^{1,+}.$ Then we introduce the vector field in the domain $\Omega^+:$
\begin{equation}\label{vd}\begin{array}{l}
{\bf b}_{+}^{(out)} (x,\varepsilon)=-\dfrac{\mathbb{F}^{(inn)}+\mathbb{F}_1^{(out)}}{2}\,
\bigg(\dfrac{\partial\tilde{\xi}_+(x,\varepsilon)}{\partial x_2},-\dfrac{\partial\tilde{\xi}_+(x,\varepsilon)}{\partial x_1}\bigg).
\end{array}
\end{equation}

\begin{lem}
The vector field ${\bf b}_{+}^{(out)}$ is solenoidal, infinitely differentiable, vanishes near $\partial\Omega^+\setminus\Lambda^{1,+}$ and $\gamma^{1,+},$ the support of ${\bf b}_{+}^{(out)}$ is contained in the set of points satisfying the inequalities $e^{-1/\varepsilon}\rho(\Delta_{\gamma^{1,+}})\leq\Delta_{\partial\Omega^+\setminus\Lambda^{1,+}}\leq \rho(\Delta_{\gamma^{1,+}}).$ Moreover,
\begin{equation}\label{int}\begin{array}{l}
\int\limits_{\Lambda^{1,+}}{\bf b}_{+}^{(out)}\cdot {\bf n}\,dS
=\dfrac{\mathbb{F}^{(inn)}+\mathbb{F}_1^{(out)}}{2},
\end{array}
\end{equation}
and the following estimates
\begin{equation}\label{bdm2}\begin{array}{l}
|{\bf b}_{+}^{(out)}(x,\varepsilon)|\leq
\dfrac{c\,\varepsilon|\mathbb{F}^{(inn)}+\mathbb{F}_1^{(out)}|}{d_{\partial\Omega^+\setminus\Lambda^{1,+}}(x)}, \ \ \ x\in\Omega^+,\\
\\
|\nabla{\bf b}_{+}^{(out)}(x,\varepsilon)|\leq
\dfrac{c\,\varepsilon|\mathbb{F}^{(inn)}+\mathbb{F}_1^{(out)}|}{d^2_{\partial\Omega^+\setminus\Lambda^{1,+}}(x)}, \ \ \ x\in\Omega^+,
\end{array}\end{equation}

\begin{equation}\label{bdm3}\begin{array}{c}
|{\bf b}_{+}^{(out)}(x,\varepsilon))|\leq
\dfrac{c(\varepsilon)
|\mathbb{F}^{(inn)}+\mathbb{F}_1^{(out)}|}{g(z^{(1)}_1)}, \ \ \ x\in D_1,\\
\\
|\nabla{\bf b}_{+}^{(out)}(x,\varepsilon)|\leq
\dfrac{c(\varepsilon)|\mathbb{F}^{(inn)}+\mathbb{F}_1^{(out)}|}{g^2(z^{(1)}_1)}, \ \ \  x\in D_1,\end{array}\end{equation}

\begin{equation}\label{bdm333}\begin{array}{c}
|{\bf b}_{+}^{(out)}(x,\varepsilon))|+|\nabla{\bf b}_{+}^{(out)}(x,\varepsilon))|\leq c(\varepsilon)|\mathbb{F}^{(inn)}+\mathbb{F}_1^{(out)}|, \ \ \ x\in {\rm supp}\,{\bf b}_{+}^{(out)}, 
\end{array}\end{equation}
hold. The constant $c$ in \eqref{bdm2} is independent of
$\varepsilon$. 
\end{lem}

\begin{proof}
The first statement of the lemma follows from the definition \eqref{vd} and from the Lemma\,\ref{xl}. 
Due to the properties of the function $\tilde{\xi}_+,$ we have
\begin{equation*}\begin{array}{l}
\int\limits_{\Lambda^{1,+}}{\bf b}_{+}^{(out)}\cdot {\bf n}\,dS=-\int\limits_{\sigma_1(R)}{\bf b}_{+}^{(out)}\cdot {\bf n}\,dS=\dfrac{\mathbb{F}^{(inn)}+\mathbb{F}_1^{(out)}}{2}\,\int\limits_{0}^{g_1(R)}\dfrac{\partial\tilde{\xi}_+}{\partial x_2}dx_2\\
\\
=\dfrac{\mathbb{F}_1^{(inn)}+\mathbb{F}_1^{(out)}}{2}\,\big(\tilde{\xi}(z^{(1)}_1,g(R),\varepsilon)-\tilde{\xi}(z^{(1)}_1,0,\varepsilon)\big)
=\dfrac{\mathbb{F}^{(inn)}+\mathbb{F}_1^{(out)}}{2}.
\end{array}
\end{equation*}
Using the inequality \eqref{xil} and the definition  \eqref{vd}, we derive the inequality \eqref{bdm2}:
\begin{equation*}\begin{array}{l}
|{\bf b}_{+}^{(out)}|\leq |\mathbb{F}^{(inn)}+\mathbb{F}_1^{(out)}|\cdot \sqrt{\sum\limits_{i=1}^2\bigg(\dfrac{\partial\tilde{\xi}_+}{\partial x_i}\bigg)^2}
\leq \dfrac{c\,\varepsilon |\mathbb{F}^{(inn)}+\mathbb{F}_1^{(out)}|}{\Delta_{\partial\Omega^+\setminus\Lambda^{1,+}}}
\leq \dfrac{c\,\varepsilon |\mathbb{F}^{(inn)}+\mathbb{F}_1^{(out)}|}{d_{\partial\Omega^+\setminus\Lambda^{1,+}}},\\
\\
|\nabla{\bf b}_{+}^{(out)}|\leq |\mathbb{F}^{(inn)}+\mathbb{F}_1^{(out)}|\cdot \sqrt{\sum\limits_{\substack{
   i,j=1 \\
   i\neq j
  }}^2\bigg(\dfrac{\partial^2\tilde{\xi}_+}{\partial x_i\partial x_j}\bigg)^2}
\leq \dfrac{c\,\varepsilon |\mathbb{F}^{(inn)}+\mathbb{F}_1^{(out)}|}{\Delta^2_{\partial\Omega^+\setminus\Lambda^{1,+}}}\\
\\
\leq \dfrac{c\,\varepsilon |\mathbb{F}^{(inn)}+\mathbb{F}_1^{(out)}|}{d^2_{\partial\Omega^+\setminus\Lambda^{1,+}}}.
\end{array}\end{equation*}
\\It remains to prove the estimates \eqref{bdm3}, \eqref{bdm333}. Since for the points $x\in {\rm supp}\, {\bf b}_{1,+}^{(out)}$ we have that $e^{-1/\varepsilon}\rho(\Delta_{\gamma^{1,+}})\leq\Delta_{\partial\Omega\setminus\Lambda^{1,+}}\leq \rho(\Delta_{\gamma^{1,+}})$ and $d_{\gamma^{1,+}}\geq\dfrac{d_0}{2},$ using the properties of the regularized distance (see the estimate \eqref{rd}) we get
\begin{equation*}\label{1}\begin{array}{l}
\dfrac{1}{a_2} e^{-1/\varepsilon}\rho(a_1\,d_{\gamma^{1,+}})\leq d_{\partial\Omega^+\setminus\Lambda^{1,+}}\leq \dfrac{1}{a_1}\rho(a_2\,d_{\gamma^{1,+}}).
\end{array}\end{equation*} 
Since $d_{\gamma^{1,+}}\geq\dfrac{d_0}{2},$ then $a_2\,d_{\gamma^{1,+}}\geq a_2\dfrac{d_0}{2}$ and \begin{equation*}\label{2}\begin{array}{l}\rho(a_2\,d_{\gamma^{1,+}})\leq \rho(2a_2\,d_{\gamma^{1,+}})=2a_2\,d_{\gamma^{1,+}}.\end{array}\end{equation*}
If $d_{\gamma^{1,+}}\geq\dfrac{a_2}{a_1}d_0,$ then $\rho(a_1\,d_{\gamma^{1,+}})=a_1\,d_{\gamma^{1,+}},$ while if $d_{\gamma^{1,+}}\leq\dfrac{a_2}{a_1}d_0,$ then
\begin{equation*}\begin{array}{l}
\rho(a_1\,d_{\gamma^{1,+}})\geq \rho(\dfrac{a_1\,d_{\gamma^{1,+}}}{2})=\dfrac{a_1\,d_0}{2}=\dfrac{a_1^2}{2a_2^2}\cdot\dfrac{a_2\,d_0}{a_1}\geq\dfrac{a_1^2}{2a_2^2}d_{\gamma^{1,+}}.
\end{array}\end{equation*}
Therefore, we have
\begin{equation*}\label{1}\begin{array}{l}
\dfrac{a_1^2\,e^{-1/\varepsilon}}{a_2^2} d_{\gamma^{1,+}}\leq\dfrac{1}{a_2} e^{-1/\varepsilon}\rho(a_1\,d_{\gamma^{1,+}})\leq d_{\partial\Omega^+\setminus\Lambda^{1,+}}\leq \dfrac{1}{a_1}\rho(a_2\,d_{\gamma^{1,+}})\leq\dfrac{2a_2}{a_1}d_{\gamma^{1,+}}.
\end{array}\end{equation*} 
\\
Since for $x\in D_1$ we have that $d_{\gamma^{1,+}}=|z^{(1)}_2-\tilde{c}|,$ where $\tilde{c}$ is the constant distance from the curve $\gamma^{1,+}$ to the $z^{(1)}_1$ axis, and $d_{\partial\Omega^+\setminus\Lambda^{1,+}}\geq c\Big(g(z^{(1)}_1)-|z^{(1)}_2-\tilde{c}|\Big)$ we obtain that
\begin{equation*}\begin{array}{l}
c\,g(z^{(1)}_1)\leq|z^{(1)}_2-\tilde{c}|\leq g(z^{(1)}_1),
\end{array}
\end{equation*}
i.e.,
\begin{equation*}\begin{array}{l}
c_1\,g(z^{(1)}_1)\leq d_{\partial\Omega^+\setminus\Lambda^{1,+}}\leq c_2\,g(z^{(1)}_1), \ \ 
x\in {\rm supp}\,{\bf b}_{+}^{(out)},
\end{array}
\end{equation*}
where $c_1$ and $c_2$ are positive constants.
\\Finally, estimates \eqref{bdm3}, \eqref{bdm333} follow from the last inequalities and from \eqref{bdm2}. 
\end{proof}

\begin{lem}\label{lfb}
For any solenoidal vector field ${\bf w}\in W^{1,2}_{loc}(\overline{\Omega^+})$ with ${\bf w}\big|_{\partial\Omega^+}=0$ the following inequalities 
\begin{equation}\label{lfbe}\begin{array}{l}
\big|\int\limits_{\Omega^+_{k+1}}({\bf w}\cdot\nabla){\bf w}\cdot{\bf b}_{+}^{(out)}\,dx\big|\leq c\,\varepsilon |\mathbb{F}^{(inn)}+\mathbb{F}_1^{(out)}|\int\limits_{\Omega^+_{k+1}}|\nabla {\bf w}|^2\,dx,\\
\\
\big|\int\limits_{\Omega^+_{k+1}\setminus\Omega^+_k}({\bf w}\cdot\nabla){\bf w}\cdot{\bf b}_{+}^{(out)}\,dx\big|\leq c\,\varepsilon |\mathbb{F}^{(inn)}+\mathbb{F}_1^{(out)}|\int\limits_{\Omega^+_{k+1}\setminus\Omega^+_k}|\nabla {\bf w}|^2\,dx
\end{array}
\end{equation}
hold, where $\Omega^+_k=\{x\in\Omega_k: x_2>0 \}$ and the constant $c$ does not depend on $\varepsilon$ and $k.$
\end{lem}

\begin{proof}
Applying the H{\"o}lder inequality and the estimates \eqref{bdm2}, \eqref{hardy1}, we obtain:
\begin{equation*}\begin{array}{l}
\big|\int\limits_{\Omega^+_{k+1}}({\bf w}\cdot\nabla){\bf w}\cdot{\bf b}_{+}^{(out)}\,dx\big|\leq \bigg(\int\limits_{\Omega^+_{k+1}}|\nabla{\bf w}|^2\,dx\bigg)^{1/2}\cdot\bigg(\int\limits_{\Omega^+_{k+1}}|{\bf w}\cdot {\bf b}_{+}^{(out)}|^2\,dx\bigg)^{1/2}\\
\\
\leq\bigg(\int\limits_{\Omega^+_{k+1}}|\nabla{\bf w}|^2dx\bigg)^{1/2}\!\!\cdot\bigg(\int\limits_{\Omega^+_{k+1}}|{\bf w}|^2\!\cdot \bigg(\dfrac{c\,\varepsilon\,|\mathbb{F}^{(inn)}+\mathbb{F}_1^{(out)}|}{d_{\partial\Omega^+\setminus\Lambda^{1,+}}}\bigg)^2dx\bigg)^{1/2}\\
\\
\leq c\,\varepsilon\,|\mathbb{F}^{(inn)}+\mathbb{F}_1^{(out)}|\int\limits_{\Omega^+_{k+1}}|\nabla{\bf w}|^2\,dx.
\end{array}
\end{equation*}
The same argument proves the second inequality of \eqref{lfbe}.
\end{proof}
Notice that Lemma\,\ref{lfb} is valid also for the domains $\Omega^-_{k+1}$ and $\Omega^-_{k+1}\setminus\Omega^-_{k}.$ Let us extend the vector field ${\bf b}_{+}^{(out)}$ into the domain $\Omega^-$ and define:
\begin{equation*}
\begin{array}{l}
{\bf b}_{1}^{(out)}(x,\varepsilon)=\begin{cases}
\big(b^{(out)}_{+,1}(x_1,x_2,\varepsilon), \ b^{(out)}_{+,2}(x_1,x_2,\varepsilon)\big), \ \ \ \ \  x\in\Omega^+,\\
\\
\big(b^{(out)}_{+,1}(x_1,-x_2,\varepsilon), \ -b^{(out)}_{+,2}(x_1,-x_2,\varepsilon)\big), \ \ \ \ \  x\in\Omega^-.
\end{cases}
\end{array}
\end{equation*}
The vector field ${\bf b}_{1}^{(out)}$ is symmetric, solenoidal, satisfies the Leray--Hopf inequality and
\begin{equation*}\begin{array}{l}
\int\limits_{\Lambda^{1}}{\bf b}_{1}^{(out)}\cdot{\bf n}\,dS=\mathbb{F}^{(inn)}+\mathbb{F}_1^{(out)}.
\end{array}
\end{equation*}
\\Let ${\bf h}^{(out)} (x,\varepsilon)=\big({\bf a} (x)-{\bf b}^{(inn)}(x)-{\bf b}_{1}^{(out)} (x,\varepsilon)\big)\big|_{\Gamma_0^1 \cap \Omega_0}.$
Then we have 
\begin{equation}\label{0}\begin{array}{l}
\int\limits_{\Gamma_0^1 \cap \Omega_0}{\bf h}^{(out)}\cdot{\bf n}\,dS=0.
\end{array}
\end{equation}
\\If ${\bf a}\in W^{1/2,2}(\partial\Omega),$ then ${\bf h}^{(out)}\in W^{1/2,2}(\partial\Omega)$ and 
\begin{equation*}\begin{array}{rcl}
\|{\bf h}^{(out)}\|_{W^{1/2,2}(\partial\Omega)} & \leq & c\big(\|{\bf a}\|_{W^{1/2,2}(\partial\Omega)}+\|{\bf b}_{1}^{(out)}\|_{W^{1/2,2}(\partial\Omega)}\big)\\
\\
& \leq & c\big(\|{\bf a}\|_{W^{1/2,2}(\partial\Omega)}+|\mathbb{F}^{(inn)}+\mathbb{F}_1^{(out)}|\big) \leq c\|{\bf a}\|_{W^{1/2,2}(\partial\Omega)}.
\end{array}\end{equation*}
Because of the condition \eqref{0} there exists (see Lemma\,\ref{extension}) an extension ${\bf b}_0^{(out)}$ of the function ${\bf h}^{(out)}$ such that ${\rm supp}\,{\bf b}_0^{(out)}(x,\varepsilon)$ is contained in a small neighborhood of $\Gamma^1_0 \cap \Omega_0$
\begin{equation}\label{b0}\begin{array}{l}
{\rm div}\,{\bf b}_0^{(out)}=0, \ \ {\bf b}_0^{(out)}(x,\varepsilon)|_{\Gamma^1_0 \cap \Omega_0}={\bf h}^{(out)}(x,\varepsilon).
\end{array}
\end{equation}
Moreover, ${\bf b}_0^{(out)}$ satisfies the Leray--Hopf inequality.
\\Notice that the vector field ${\bf b}^{(out)}_0$ is not necessary symmetric. However, since the boundary value ${\bf h}^{(out)}$ is symmetric, ${\bf b}^{(out)}_0$ can be symmetrized to $\tilde{{\bf b}}^{(out)}_0=(\tilde{b}^{(out)}_{0,1},\tilde{b}^{(out)}_{0,2})$  as in \eqref{symm}.
Then set
\begin{equation}\label{b}\begin{array}{l}
{\bf B}_{1}^{(out)} (x,\varepsilon)={\bf b}_{1}^{(out)} (x,\varepsilon)+\tilde{{\bf b}}^{(out)}_0 (x,\varepsilon).
\end{array}
\end{equation}
The following lemma is a direct corollary of the previous lemmas.
\begin{lem}\label{out}
The vector field ${\bf B}_{1}^{(out)}\in W^{1,2}_{loc}(\Omega)$ is symmetric and solenoidal, ${\bf B}_{1}^{(out)}\big|_{\Gamma^1_0 \cap \Omega_0}={\bf a}\big|_{\Gamma^1_0 \cap \Omega_0}-{\bf b}^{(inn)}\big|_{\Gamma_0^1 \cap \Omega_0},$  ${\bf B}_{1}^{(out)}\big|_{\partial\Omega\setminus(\Gamma^1_0 \cap \Omega_0)}=0.$ Further, for any solenoidal symmetric vector field ${\bf w}\in W^{1,2}_{loc}(\overline{\Omega})$ with ${\bf w}\big|_{\partial\Omega}=0$ the following inequalities
\begin{equation}\label{oute}\begin{array}{l}
\big|\int\limits_{\Omega_{k+1}}({\bf w}\cdot\nabla){\bf w}\cdot{\bf B}_{1}^{(out)}\,dx\big|\leq c\,\varepsilon |\mathbb{F}^{(inn)}+\mathbb{F}_1^{(out)}|\int\limits_{\Omega_{k+1}}|\nabla {\bf w}|^2\,dx,\\
\\
\big|\int\limits_{\Omega_{k+1}\setminus\Omega_k}({\bf w}\cdot\nabla){\bf w}\cdot{\bf B}_{1}^{(out)}\,dx\big|\leq c\,\varepsilon |\mathbb{F}^{(inn)}+\mathbb{F}_1^{(out)}|\int\limits_{\Omega_{k+1}\setminus\Omega_k}|\nabla {\bf w}|^2\,dx
\end{array}
\end{equation}
hold. The constant $c$ does not depend on $\varepsilon$ and $k.$
Moreover,
\begin{equation}\label{outg}\begin{array}{l}
|{\bf B}_{1}^{(out)}(x,\varepsilon)|\leq \dfrac{C(\varepsilon)|\mathbb{F}^{(inn)}+\mathbb{F}_1^{(out)}|}{g_j(z^{(j)}_1)}, \ \ \ x\in D_j, \ \ j=1,2,
\\
\\
|\nabla{\bf B}_{1}^{(out)}(x,\varepsilon)|\leq \dfrac{C(\varepsilon)|\mathbb{F}^{(inn)}+\mathbb{F}_1^{(out)}|}{g_j^2(z^{(j)}_1)}, \ \ \ x\in D_j, \ \ j=1,2,\\
\\
|{\bf B}_{1}^{(out)}(x,\varepsilon)|+|\nabla{\bf B}_{1}^{(out)}(x,\varepsilon)|\leq 
C(\varepsilon)|\mathbb{F}^{(inn)}+\mathbb{F}_1^{(out)}|, \ \ \ x\in {\rm supp}\,{\bf B}_{1}^{(out)} .
\end{array}
\end{equation}
\end{lem}
\begin{remark}\label{outr}
In the same manner we can construct the extension ${\bf B}_2^{(out)}$ and the previous lemma remains valid for this extension with very slight differences. Firstly, ${\bf B}_{2}^{(out)}\big|_{\Gamma^2_0 \cap \Omega_0}={\bf a}\big|_{\Gamma^2_0 \cap \Omega_0}.$ Secondly, the estimates \eqref{oute} and \eqref{outg} holds for ${\bf B}_2^{(out)}$ if instead of $|\mathbb{F}^{(inn)}+\mathbb{F}_1^{(out)}|$ we take $|\mathbb{F}_2^{(out)}|.$
\end{remark}
According to Lemma\,\ref{out} and Remark\,\ref{outr} the constructed extension ${\bf B}^{(out)}$ has form \eqref{B}. Therefore, the vector field
$$\begin{array}{l}
{\bf A}={\bf B}^{(inn)}+{\bf B}^{(out)}
\end{array}$$
has all the necessary properties that insure the validity of the Theorem\,3.1.

\subsection{$\mathbb{Y}$ Type Domain}

In this case we consider the domain $
\Omega = \Omega_0\cup D_1\cup D_2\cup D_3,
$ 
where $D_1$ is a self-symmetric outlet while $D_2,$ $D_3$ is a pair of symmetric outlets (see Fig.\,3). The outer boundary consists of three disjoint connected unbounded components $\Gamma_0^1,$ $\Gamma_0^2$ and $\Gamma_0^3.$ 
\\Since $\Lambda^2\subset\Gamma_0^2$ and $\Lambda^3\subset\Gamma_0^3$ are symmetric, we have 
$$\begin{array}{l}
\int\limits_{\Lambda^1} {\bf a }\cdot {\bf n} \, dS=\mathbb{F}_1^{(out)},
\int\limits_{\Lambda^2} {\bf a }\cdot {\bf n} \, dS=\dfrac{1}{2}\mathbb{F}_2^{(out)}, \ \ \int\limits_{\Lambda^3} {\bf a }\cdot {\bf n} \, dS=\dfrac{1}{2}\mathbb{F}_2^{(out)}.
\end{array}$$ 
As usual $\int\limits_{\Gamma_i}{\bf a}\cdot {\bf n}\,dS=\mathbb{F}_i^{(inn)}, \ \ i=1,...,I.$ 
\\For the correct formulation we have to prescribe the fluxes over each outlet having in mind that $D_2$ and $D_3$ is a pair of symmetric outlets, i.e.,
$$\begin{array}{l}\int\limits_{\sigma_1(R)} {\bf u}\cdot {\bf n}\,dS=\mathbb{F}_1, \ \ 
\int\limits_{\sigma_2(R)} {\bf u}\cdot {\bf n}\,dS=\dfrac{1}{2}\mathbb{F}_2, \ \ 
\int\limits_{\sigma_3(R)} {\bf u}\cdot {\bf n}\,dS=\dfrac{1}{2}\mathbb{F}_2, \ \ R>R_0>0.\end{array}$$
\\Then we have that the necessary compatibility condition \eqref{nc1} can be rewritten as
$$ \mathbb{F}^{(inn)}+\mathbb{F}^{(out)}+\mathbb{F}_1+\mathbb{F}_2=0,$$
where $$\mathbb{F}^{(inn)}=\sum\limits_{i=1}^I \mathbb{F}_i^{(inn)}, \ \ \mathbb{F}^{(out)}=\sum\limits_{m=1}^3 \mathbb{F}_m^{(out)}.$$
\begin{figure}[ht!]
\centering
\includegraphics[scale=0.35]{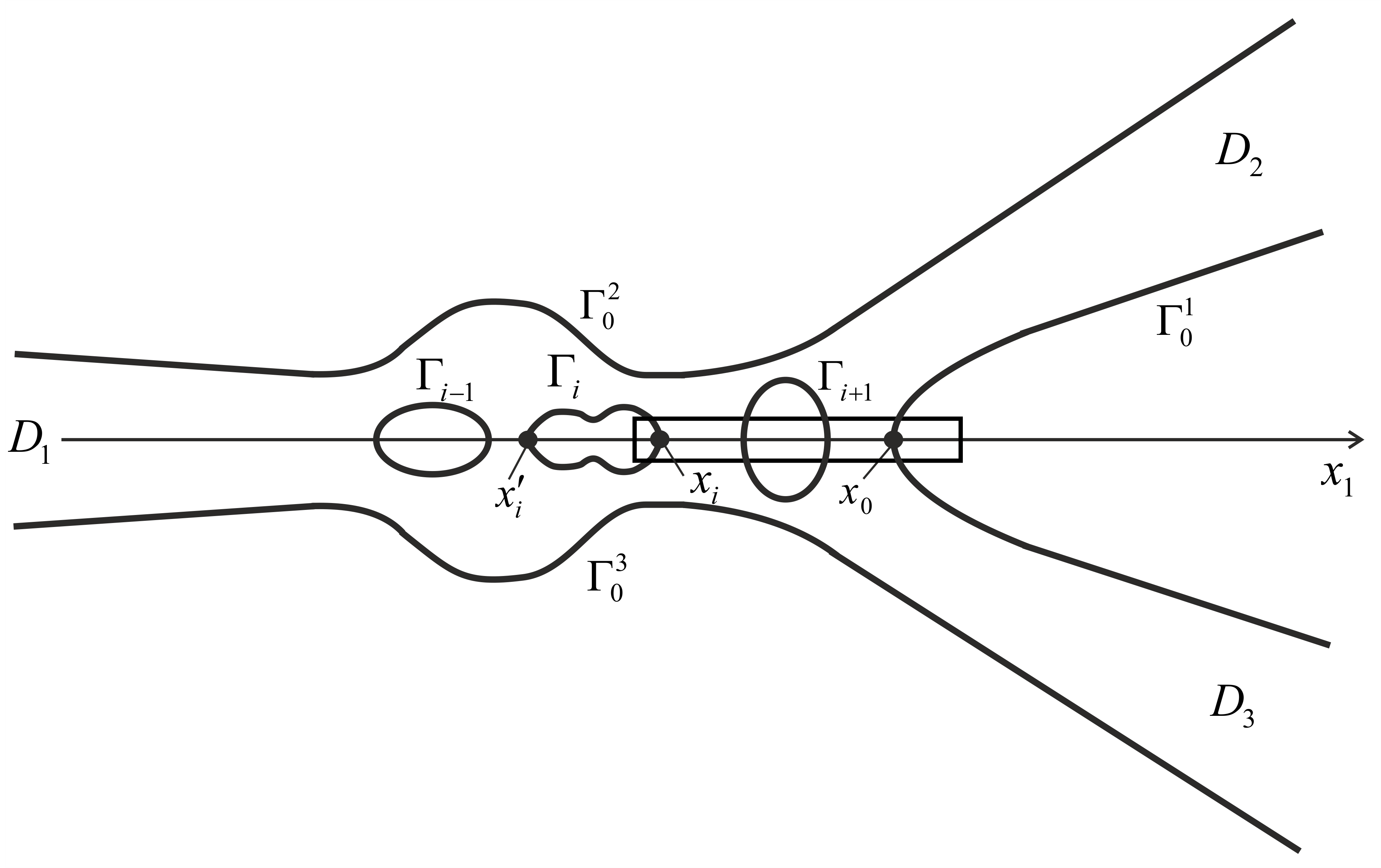}
\caption{Domain $\Omega$}
\end{figure}

\subsubsection{Construction of the Extension}

The extension ${\bf B}^{(inn)}$ is constructed in the same way as in the subsection\,4.1.1. We remove the fluxes from the inner boundaries $\Gamma_i, \ \ i=1,...,I,$ to the outer boundary $\Gamma_0^1.$ Next step is to construct the extensions from the outer boundaries. More precisely, we construct ${\bf B}_1^{(out)}$ which extends the modified\footnote{Let us remind that ${\bf b}^{(inn)}$ is a virtual drain function which removes the fluxes from the inner boundaries $\Gamma_1,...,\Gamma_I$ to the outer boundary $\Gamma^1_0.$} boundary value ${\bf a}-{\bf b}^{(inn)}$ from the connected component $\Gamma_0^{1}$
and ${\bf B}_2^{(out)}$ which extends the origin boundary value ${\bf a}$ from the  connected symmetric components  $\Gamma_0^{2}$ and $\Gamma_0^{3}.$ Hence, the extension of the outer boundary $\partial\Omega\setminus\cup_{i=1}^I \Gamma_i$ has to be constructed as the sum
$${\bf B}^{(out)}={\bf B}_1^{(out)}+{\bf B}_2^{(out)}.$$
This construction is analogous to that of the subsection\,4.1.2.
\\The main difference for the formulation of the problem in $\mathbb{V}$ and $\mathbb{Y}$ type domains is that for the case of $\mathbb{Y}$ type domain we have to prescribe the fluxes over the cross sections of the outlets, i.e., we need to construct, in addition, the flux carrier ${\bf B}^{(flux)},$ which has zero boundary value over $\partial\Omega$ and removes the fluxes over the cross sections of the outlets to infinity. We start with the construction of the auxiliary function ${\bf b}^{(flux)}.$ Let us introduce an infinite simple curve $\gamma_+\subset\Omega^+,$ consisting of two semi-infinite lines $\gamma_{1}\subset D_{1},$ $\gamma_{3}\subset D_{3}$ and finite curve $\gamma_0\subset\Omega_0.$
Notice that $\gamma_+$ does not intersect the boundary $\partial\Omega$ and the direction of this curve coincides with the increase of the coordinate $z^{(1)}_1.$ Then we can define in $\Omega^+$ a cut-off function
\begin{equation*}\begin{array}{l}
\xi_+(x)=\Psi\Big(\varepsilon\,\ln\dfrac{\rho(\Delta_{\gamma_+}(x))}{\Delta_{\partial\Omega^+}(x)}\Big),
\end{array}\end{equation*}
where functions $\rho$ and $\Psi$ are defined by \eqref{ro} and \eqref{fi}, respectively.
\\Since curve $\gamma_+$ decomposes $\Omega^+$ into two parts situated above and below $\gamma_+,$ we define
$\tilde{\xi}_+(x,\varepsilon)=\xi_+(x,\varepsilon)$ above  $\gamma_+$ and $\tilde{\xi}_+(x,\varepsilon)=0$ below  $\gamma_+.$
Thus, we introduce 
\begin{equation*}\begin{array}{l}
{\bf b}_+(x)=\Big(\dfrac{\partial\tilde{\xi}_+(x)}{\partial x_2},\,-\dfrac{\partial\tilde{\xi}_+(x)}{\partial x_1} \Big),  \ \ \ x\in\Omega^+.
\end{array}\end{equation*}
\begin{lem}\label{bl}
The vector field ${\bf b}_+$ is solenoidal, vanishes on the boundary $\partial\Omega^+$ and satisfies
\begin{equation}\label{b1}\begin{array}{l}
\int\limits_{\sigma_2(R)}{\bf b}_+\cdot {\bf n}\,dS=1, \ \ \ \ \int\limits_{\sigma_1(R)\cap\Omega^+}{\bf b}_+\cdot {\bf n}\,dS=-1, \ \ \ R>R_0.
\end{array}\end{equation}
\\For any solenoidal vector field ${\bf w}\in W_{loc}^{1,2}(\overline{\Omega^+})$ with ${\bf w}\big|_{\partial\Omega^+}=0$ the following Leray-Hopf inequality
\begin{equation}\label{b2}\begin{array}{l}
\big|\int\limits_{\Omega^+_{k+1}}({\bf w}\cdot\nabla)\,{\bf w}\cdot {\bf b}_+\,dx\big|\leq c\,\varepsilon\int\limits_{\Omega^+_{k+1}}|{\bf w}|^2dx,\\
\\
\big|\int\limits_{\Omega^+_{k+1}\setminus\Omega^+_k}({\bf w}\cdot\nabla)\,{\bf w}\cdot {\bf b}_+\,dx\big|\leq c\,\varepsilon\int\limits_{\Omega^+_{k+1}\setminus\Omega^+_k}|{\bf w}|^2dx
\end{array}\end{equation}
hold, where the constant $c$ does not depend on $\varepsilon$ and 
$k.$ Moreover,
\begin{equation}\label{b3}\begin{array}{l}
\big|{\bf b}_+(x)\big|\leq\dfrac{C(\varepsilon)}{g^2_{i}(z_1^{(i)})}, \ \ \ \big|\nabla{\bf b}_+(x)\big|\leq\dfrac{C(\varepsilon)}{g^3_i(z_1^{(i)})},\ \ \ \ x\in D_1\cap\Omega^+, \ D_2, \ \ i=1,2,\\
\\
\big|{\bf b}_+(x)\big|+\big|\nabla{\bf b}_+(x)\big|\leq C(\varepsilon), \ \ \ x\in {\rm supp}\,{\bf b}_+
\end{array}\end{equation}
hold.
\end{lem}
Let us define
\begin{equation*}\begin{array}{l}
{\bf B}^{(flux)}_+=\dfrac{\mathbb{F}_2+\mathbb{F}^{(inn)}+\mathbb{F}^{(out)}}{2}\,{\bf b}_+.
\end{array}\end{equation*}
\\Then from the previous lemma it follows that
\begin{equation*}\begin{array}{l}
\int\limits_{\sigma_2(R)} {\bf B}^{(flux)}_+\cdot {\bf n}\,dS=\dfrac{\mathbb{F}_2+\mathbb{F}^{(inn)}+\mathbb{F}^{(out)}}{2}, \ \ \ \ \int\limits_{\sigma_1(R)\cap\Omega^+} {\bf B}^{(flux)}_+\cdot {\bf n}\,dS=\dfrac{\mathbb{F}_1}{2}, \ \ \ R>R_0.
\end{array}\end{equation*}
Let us extend the vector field ${\bf B}^{(flux)}_+$ into the domain $\Omega^-$ and define
\begin{equation*}\begin{array}{l}
{\bf B}^{(flux)}(x)=\begin{cases}
\big(B_{+,1}^{(flux)}(x_1,x_2), \, B_{+,2}^{(flux)}(x_1,x_2)\big), \ \ \ x\in\Omega^+,\\
\\
\big(B_{+,1}^{(flux)}(x_1,-x_2), \, B_{+,2}^{(flux)}(x_1,-x_2)\big), \ \ \ x\in\Omega^-.
\end{cases}
\end{array}\end{equation*}
The vector field ${\bf B}^{(flux)}$ is symmetric, solenoidal, satisfies the Leray-Hopf inequality and the following flux conditions
\begin{equation}\label{b4}\begin{array}{l}
\int\limits_{\sigma_1(R)}{\bf B}^{(flux)}\cdot {\bf n}\,dS=\mathbb{F}_1,
\\
\\
\int\limits_{\sigma_2(R)}{\bf B}^{(flux)}\cdot {\bf n}\,dS+\int\limits_{\sigma_3(R)}{\bf B}^{(flux)}\cdot {\bf n}\,dS=\mathbb{F}_2+\mathbb{F}^{(inn)}+\mathbb{F}^{(out)}.
\end{array}\end{equation}
Therefore, the vector field
\begin{equation*}\begin{array}{l}
{\bf A}={\bf B}^{(inn)}+{\bf B}^{(out)}+{\bf B}^{(flux)}
\end{array}\end{equation*}
gives the desired extension of the boundary value ${\bf a}$
which has all the necessary properties that insure the validity of the Theorem\,3.1.

\subsection{$\mathbb{I}$ type outlet}

In this section we study the domain $\Omega=\Omega
_0\cup D_1\cup D_2,$ where $D_1$ and $D_2$
are self-symmetric outlets (see Fig.\,4). The outer boundary consists of two disjoint connected unbounded components $\Gamma_0^1$ and $\Gamma_0^2.$ Since $\Lambda^1\subset\Gamma_0^1$ and $\Lambda^2\subset\Gamma_0^2$ are symmetric, the fluxes over the outer boundaries are
$$\begin{array}{l}
\int\limits_{\Lambda^1}{\bf a}\cdot {\bf n}\,dS=\dfrac{1}{2}\mathbb{F}_1^{(out)}, \ \ \ 
\int\limits_{\Lambda^2}{\bf a}\cdot {\bf n}\,dS=\dfrac{1}{2}\mathbb{F}_1^{(out)} 
\end{array}$$
and the fluxes over the inner boundaries are as usual
$$\begin{array}{l} 
\int\limits_{\Gamma_i}{\bf a}\cdot {\bf n}\,dS=\mathbb{F}_i^{(inn)}, \ \ \ i=1,...,I.
\end{array}$$
Moreover, we have to prescribe the fluxes over the cross sections of the outlets:
$$\begin{array}{l}\int\limits_{\sigma_1(R)} {\bf u}\cdot {\bf n}\,dS=\mathbb{F}_1, \ \ 
\int\limits_{\sigma_2(R)} {\bf u}\cdot {\bf n}\,dS=\mathbb{F}_2, \ \ R>R_0>0.\end{array}$$
\\Then we have that the necessary compatibility condition \eqref{nc1} can be rewritten as
$$ \mathbb{F}^{(inn)}+\mathbb{F}^{(out)}+\mathbb{F}_1+\mathbb{F}_2=0,$$
where $$\mathbb{F}^{(inn)}=\sum\limits_{i=1}^I \mathbb{F}_i^{(inn)}, \ \ \mathbb{F}^{(out)}=\sum\limits_{m=1}^2 \mathbb{F}_m^{(out)}.$$

\subsubsection{Construction of the Extension}

Notice that the outer boundary does not intersect the $x_1$ axis, i.e., the previous method how we removed the fluxes from the inner boundaries to one of the outer boundaries does not work any more. Therefore, we have slightly modify the construction of the extension ${\bf A},$ i.e., instead of the vector field ${\bf B}^{(inn)}$ we construct ${\bf B}_0+{\bf B}_{\infty}$ (using the same technique as in \cite{KW}), where ${\bf B}_0$ extends the boundary value ${\bf a}$ from $\Gamma_1,..., \Gamma_{I-1},$ and ${\bf B}_{\infty}$ extends ${\bf a}$ from $\Gamma_I.$ Thus, ${\bf A}$ is constructed as the sum:
\begin{equation}\label{EI}\begin{array}{l}
{\bf A}={\bf B}_0+{\bf B}_{\infty}+{\bf B}^{(out)}_1+{\bf B}^{(flux)},
\end{array}\end{equation}
where the vector fields ${\bf B}_1^{(out)}$ and ${\bf B}^{(flux)}$ are constructed in the same way as in the previous section\,4.2.
\begin{figure}[ht!]
\centering
\includegraphics[scale=0.35]{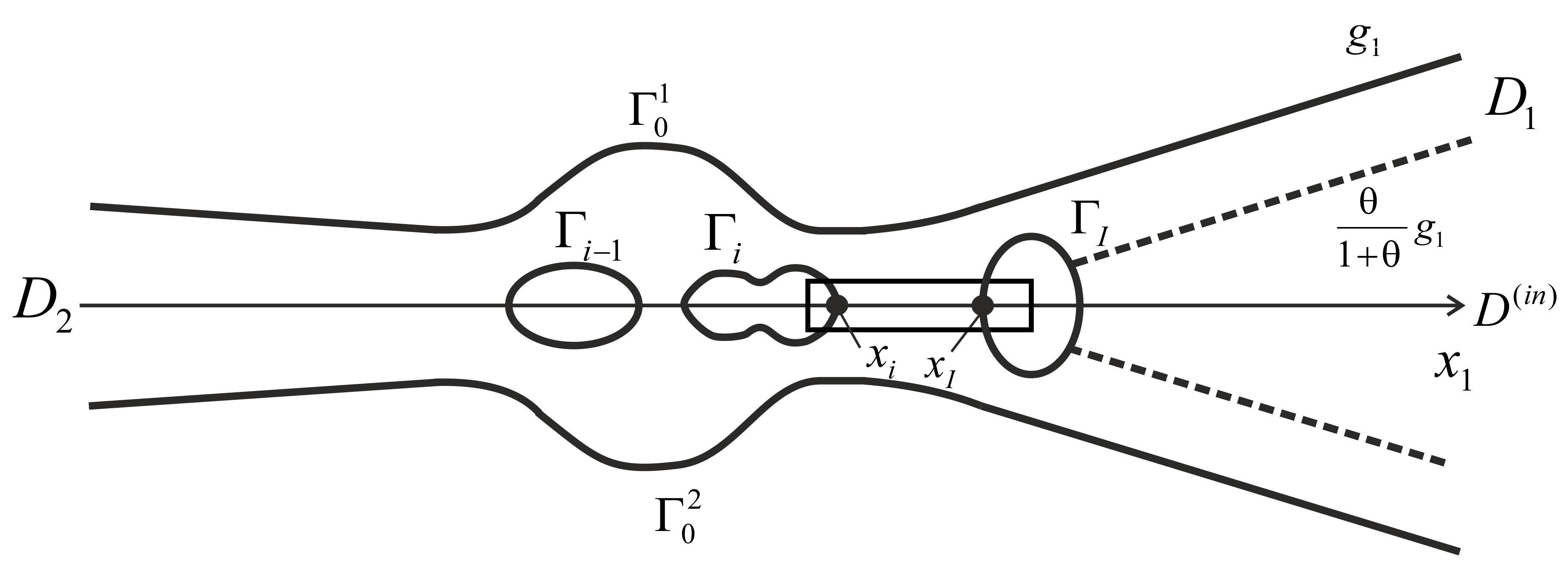}
\caption{Domain $\Omega$}
\end{figure}

In order to construct ${\bf B}_0$ we have to remove the fluxes $\mathbb{F}_i^{(inn)}$, $i=1,\cdots I-1,$ from $\Gamma_1,..., \Gamma_{I-1}$ to the last inner boundary $\Gamma_I$ (with the help of auxiliary function ${\bf b}_i$) and then extend the modified boundary value ${\bf a}-\sum\limits_{i=1}^{I-1}{\bf b}_i$ which has zero fluxes on $\Gamma_i, \ i=1,...,I-1,$ into $\Omega.$ 
Similarly to $\mathbf{b}^{(inn)}$ (constructed in subsection 4.1.1) we define ${\bf b}_i$ on each strip $\Upsilon_i \cap \Omega$, $i=1, \dots, I-1,$ joining $\Gamma_i$ and $\Gamma_I$ (see Fig.\,4) as follows:
\begin{equation}
{\bf b}_i(x)=\begin{cases}
-\dfrac{\mathbb{F}^{(inn)}_i}{2}\big(\beta(x_2),0)\big), \quad {\rm in} \quad \Upsilon_i\cap \Omega, \; x_2>0,\\
\\
-\dfrac{\mathbb{F}^{(inn)}_i}{2}\big(\beta(-x_2),0)\big), \quad {\rm in} \quad \Upsilon_i\cap \Omega, \; x_2<0,\\
\\
(0,0), \ \ \overline{\Omega}\setminus(\Upsilon_i\cap \Omega),
\end{cases}
\end{equation}
where $\beta(x_2)=-\dfrac{\partial \xi_\delta(x_2)}{\partial x_2}$ and $\xi_\delta(x)=\Psi \big(\varepsilon \ln \dfrac{\rho(\Delta_\delta)}{\Delta_\gamma}\big)$ with $\Psi$ and $\rho$ defined by \eqref{fi} and \eqref{ro}, respectively, $\Delta_\delta$ denotes the regularized distance function to the line $x_2=\delta$ and $\gamma$ denotes $x_2=0$.
\\Notice that 
$$\begin{array}{l}
\int\limits_{\Gamma_i}{\bf b}_i\cdot {\bf n}\,dS =\mathbb{F}_i^{(inn)}, \ \ \forall i=1,\dots, I-1.
\end{array}$$ 
Set
$${\bf b}= \sum\limits_{i=1}^{I-1}{\bf b}_i.$$
Thus  ${\bf b}$ is a symmetric solenoidal vector field. Moreover for every $i=1,...,I-1$ one has (note that the flux of ${\bf b}_i$ vanishes on $\Gamma_j$ for every $i\neq j$)

\begin{equation}\label{zerooo}\begin{array}{l}
\int\limits_{\Gamma_i}({\bf a}-{\bf b})\cdot {\bf n}\,dS=\int\limits_{\Gamma_i}({\bf a}-{\bf b}_i)\cdot {\bf n}\,dS=\mathbb{F}^{(inn)}_i-\mathbb{F}^{(inn)}_i=0.
\end{array}\end{equation}
Because of \eqref{zerooo}, 
there exists a solenoidal extension ${\bf A}_0$ of $\big({\bf a}-{\bf b}\big)\Big|_{\cup_{i=1}^{I-1}\Gamma_i}$ satisfiying the Leray--Hopf inequality. Notice that ${\bf A}_0$ can be symmetrized to $\tilde{\bf A}_0$ as in \eqref{symm}.
\\Define then
$$\mathbf{B}_0=\tilde{\bf A}_0+{\bf b}.$$
${\bf B}_0$ is a symmetric extension of the boundary value ${\bf a}$ from $\cup_{i=1}^{I-1}\Gamma_i$. Moreover, 
note that on the last inner boundary $\Gamma_I$ we have
\begin{equation*}\begin{array}{l}
\int\limits_{\Gamma_I}({\bf a}-{\bf b})\cdot {\bf n}\,dS= \sum\limits_{i=1}^{I} \mathbb{F}_i^{(inn)}=\mathbb{F}^{(inn)}.
\end{array}\end{equation*}
In \cite{KW} we have proved the following lemma.
\begin{lem}\label{binnl} Assume that the boundary value ${\bf a}$ is a symmetric function in $W^{1/2,2}(\partial\Omega)$ having a compact support. Denote by $\tilde{{\bf a}}$ the restriction of ${\bf a}$ to $\cup_{i=1}^{I-1}\Gamma_i.$ Then for every $\varepsilon > 0$ there exists a symmetric solenoidal extension ${\bf B}_0$ in $\Omega$ satisfying ${\bf B}_0\Big|_{\cup_{i=1}^{I-1}\Gamma_i}=\tilde{{\bf a}},$ ${\bf B}_0\Big|_{\partial\Omega\setminus\cup_{i=1}^{I-1}\Gamma_i}=0$
and the Leray-Hopf inequality, i.e., for every symmetric solenoidal function ${\bf w} \in W^{1,2}_{loc}(\Omega)$ with ${\bf w}|_{\partial\Omega}=0$ the following estimates
\begin{equation}\label{binn}\begin{array}{l}
\big|\int\limits_{\Omega_{k+1}}({\bf w} \cdot \nabla){\bf w} \cdot {\bf B}_0\, dx\big| \leq c\,\varepsilon\, \int\limits_{\Omega_{k+1}}|\nabla {\bf w}|^2\,dx
\end{array}\end{equation}
hold.
\end{lem}
After constructing ${\bf B}_0$ we have the flux $\sum\limits_{i=1}^I\mathbb{F}_i^{(inn)}=\mathbb{F}^{(inn)}$ on $\Gamma_I.$ 
Note that ${\bf B}_0$ extends the boundary value ${\bf a}$ only from $\Gamma_i,$ $i=1,\cdots, I-1,$ and is equal to ${\bf a}+{\bf b}$ on $\Gamma_I.$ Moreover,  
$\int\limits_{\Gamma_I}{\bf B}_0\cdot \nn\,dS=\sum\limits_{i=1}^I\mathbb{F}_i^{(inn)}.$
In order to construct $\bf{B}_\infty,$ we embed a smaller outlet $D^{(in)}$ into $D_1$ in such a way that the curve $x_2=\dfrac{\theta}{\theta +1} g(x_1)$ crosses the last boundary $\Gamma_I$ for a chosen positive number $\theta$ (see Fig.\,4).
\\Then set
$$\bf{b}_\infty=
\begin{cases}
\dfrac{\mathbb{F}^{(inn)}}{2}\Big( \dfrac{\partial \xi(x_1,x_2)}{\partial x_2}, -\dfrac{\partial\xi(x_1,x_2)}{\partial x_1}\Big), \quad x_2>0,\\
\\
\dfrac{\mathbb{F}^{(inn)}}{2}\Big( \dfrac{\partial \xi(x_1,-x_2)}{\partial x_2}, -\dfrac{\partial\xi(x_1,-x_2)}{\partial x_1}\Big), \quad x_2<0,
\end{cases}
$$
where  $\xi(x)=\Psi\big(\varepsilon\ln \dfrac{\rho(\Delta_{\frac{\theta}{1+\theta}g})}{\Delta_\gamma} \big)$ for $x\in D^{(in)}$ and $\xi$ is extended by $0$ into $D_1$.
\\Since for any cross section $\sigma_1(R)$ it holds that $\int\limits_{\Gamma_I}{\bf b}_\infty \cdot {\bf n}\,dS=-\int\limits_{\sigma_1(R)}{\bf b}_\infty \cdot {\bf n}\,dS= \sum\limits_{i=1}^I\mathbb{F}_i^{(inn)},$ one has
$\int\limits_{\Gamma_I}({\bf a}-{\bf b}-{\bf b}_\infty)\cdot\,{\bf n}\,dS=0.$ Because of the last condition there exists a solenoidal
extension $\bf{A}_\infty$ of $({\bf a}-{\bf b}-{\bf b}_\infty)|_{\Gamma_I}$ satisfying the Leray-Hopf inequality. Notice that ${\bf A}_{\infty}$ can be symmetrized to $\tilde{\bf A}_\infty$ as in \eqref{symm}. Then $\bf{B}_\infty = {\bf b}+\bf{b}_\infty+\tilde{\bf A}_\infty.$ Set  ${\bf B}_0+{\bf B}_\infty.$ 
The proof of the following lemma can be found in \cite{KW}. 
\begin{lem}\label{binfty}  Assume that the boundary value ${\bf a}$ is a symmetric function in $W^{1/2,2}(\partial\Omega)$ having a compact support. For every $\varepsilon > 0$ there exists a symmetric solenoidal extension ${\bf B}_\infty$ of  ${\bf a}|_{\Gamma_I}$ in $\Omega$ satisfying the Leray-Hopf inequality, i.e., for every symmetric solenoidal function ${\bf w} \in W^{1,2}_{loc}(\Omega)$ with ${\bf w}|_{\partial\Omega}=0$ the following estimates
\begin{equation}\label{binn}\begin{array}{l}
\big|\int\limits_{\Omega_{k+1}}({\bf w} \cdot \nabla){\bf w} \cdot ({\bf B}_0+{\bf B}_\infty)\, dx\big| \leq c\,\varepsilon\, \int\limits_{\Omega_{k+1}}|\nabla {\bf w}|^2\,dx\\
\\
\big|\int\limits_{\Omega_{k+1}\setminus\Omega_k}({\bf w} \cdot \nabla){\bf w} \cdot ({\bf B}_0+{\bf B}_\infty)\, dx\big| \leq c\,\varepsilon\, \int\limits_{\Omega_{k+1}\setminus\Omega_k}|\nabla {\bf w}|^2\,dx
\end{array}\end{equation}
hold.
\end{lem}
The vector fields ${\bf B}^{(out)}_1$ and ${\bf B}^{(flux)}$ are constructed analogically as in the previous subsections 4.1.2 and 4.2.1, respectively. 

Therefore, we have an extension of the form \eqref{EI} which
has all the necessary properties and insures the validity of the Theorem\,3.1.

\begin{remark}
Combining general method ( $\mathbb{V}$ and $\mathbb{Y}$ types domains) and slightly different method ($\mathbb{I}$ type domain) we can construct an extension ${\bf A}$ of the boundary value ${\bf a}$ when the symmetric domain $\Omega$ has finitely many outlets to infinity.
\end{remark}

\section*{Acknowledgement}

The research leading to these results has received funding from Lithuanian-Swiss cooperation programme to reduce economic and social disparities within the enlarged European Union under project agreement No. CH-3-SMM-01/01.

 \end{document}